\documentclass[twoside]{article}

\usepackage[accepted]{aistats_preprint}

\usepackage{amsmath}
\usepackage{amsfonts}
\usepackage{graphicx}
\usepackage{amsopn}

\usepackage{amssymb}

\usepackage[round]{natbib}

\usepackage[textsize=tiny]{todonotes}

\usepackage{amsmath,amssymb}
\usepackage{xcolor}
\usepackage[utf8]{inputenc} 
\usepackage[T1]{fontenc}    
\usepackage{hyperref}       
\usepackage{url}            
\usepackage{booktabs}       
\usepackage{amsfonts}       
\usepackage{nicefrac}       
\usepackage{microtype}      
\usepackage{amsthm}
\usepackage{graphicx}
\usepackage{subcaption}

\usepackage{caption}

\usepackage{algorithm}
\usepackage{algorithmic}
\usepackage[Symbolsmallscale]{upgreek}
\DeclareMathOperator*{\argmin}{arg\,min}
\DeclareMathOperator*{\argmax}{arg\,max}
\newcommand{\eg}[1]{\textcolor{red}{#1}}

\def\eg#1{{\color{black}#1}}

\colorlet{lightpurple}{purple!30}
\usepackage{makecell}
\usepackage{pifont}

\theoremstyle{plain}
\newtheorem{theorem}{Theorem}[section]
\newtheorem{lemma}[theorem]{Lemma}
\newtheorem{corollary}[theorem]{Corollary}
\newtheorem{proposition}[theorem]{Proposition}

\theoremstyle{definition}
\newtheorem{definition}[theorem]{Definition}

\newtheorem{assumption}[theorem]{Assumption}

\theoremstyle{remark}
\newtheorem{remark}{Remark}

\def\L{\mathcal{L}}
\def\R{\mathbb{R}}
\newcommand{\smallfrac}[2]{{\textstyle\frac{#1}{#2}}}
\newcommand{\relu}[1]{\left[{\textstyle #1}\right]_+}
\newcommand{\negrelu}[1]{\left[-{\textstyle #1}\right]_+}

\def\Sc{\mathcal{S}}
\def\Ac{\mathcal{A}}
\def\Hc{\mathcal{H}}

\def\H{\mathcal{H}}
\def\O{\mathcal{O}}

\newcommand{\norm}[1]{\left\|{#1}\right\|} 
\newcommand{\normbig}[1]{\big\|{#1}\big\|} 

\usepackage{physics}

\begin{document}

\twocolumn[

\aistatstitle{Algorithm for Constrained Markov Decision Process with Linear Convergence}

\aistatsauthor{
Egor Gladin \\ \small{\textsf{egor.gladin@student.hu-berlin.de}} \\
{\fontsize{8}{10}\selectfont Humboldt University of Berlin,} \\ {\fontsize{8}{10}\selectfont Moscow Institute of Physics and Technology}
\And
Maksim Lavrik-Karmazin \\ \small{\textsf{lavrik-karmazin.mb@phystech.edu}} \\ {\fontsize{8}{10}\selectfont Moscow Institute of Physics and Technology} \\ \And
Karina Zainullina \\ \small{\textsf{zaynullina.ke@phystech.edu}} \\ {\fontsize{8}{10}\selectfont Moscow Institute of Physics and Technology} \AND
Varvara Rudenko \\ \small{\textsf{rudenko.vd@phystech.edu}} \\ {\fontsize{8}{10}\selectfont Moscow Institute of Physics and Technology,} \\ {\fontsize{8}{10}\selectfont HSE University} \And
Alexander Gasnikov \\ \small{\textsf{gasnikov@yandex.ru}} \\ {\fontsize{8}{10}\selectfont Moscow Institute of Physics and Technology,} \\ {\fontsize{8}{10}\selectfont ISP RAS Research Center}\\ {\fontsize{8}{10}\selectfont for Trusted Artificial Intelligence} \And
Martin Takáč \\ \small{\textsf{takac.mt@gmail.com}} \\ {\fontsize{8}{10}\selectfont Mohamed bin Zayed University} \\ {\fontsize{8}{10}\selectfont of Artificial Intelligence}}

\aistatsaddress{} ]
\begin{abstract}
    The problem of constrained Markov decision process is considered. An agent aims to maximize the expected accumulated discounted reward subject to multiple constraints on its costs (the number of constraints is relatively small). A new dual approach is proposed with the integration of two ingredients: entropy-regularized policy optimizer and Vaidya’s dual optimizer, both of which are critical to achieve faster convergence. The finite-time error bound of the proposed approach is provided. Despite the challenge of the nonconcave objective subject to nonconcave constraints, the proposed approach is shown to converge (with linear rate) to the global optimum. The complexity expressed in terms of the optimality gap and the constraint violation significantly improves upon the existing primal-dual approaches.
\end{abstract}
\section{Introduction}
In this paper we consider $\gamma$-discounted infinite-horizon constrained Markov decision process (CMDP) \cite{altman1999constrained}. Such problem arises in many practical applications, such as autonomous driving \cite{fisac2018general}, robotics \cite{ono2015chance} 
or 
systems where the agent must meet safety constraints.
An example of such a problem is an energy-efficient wireless communication system that aims to consume minimum power without violating any constraint on quality service \cite{li2016cmdp}. Such Reinforcement Learning (RL) problems are often formulated as CMDP \cite{garcia2015comprehensive}.

Recently, \cite{ying2022dual,lanarcpo,liu2021fast} proposed algorithms  (under various assumptions)
that achieve $\tilde{\mathcal{O}}\left(1/\epsilon\right)$\footnote{For clarity we skip the dependence on $1 - \gamma$ and logarithmic factors.} iteration complexity to find global optimum, where $\epsilon$ characterizes optimality gap and constraint violation. Each iteration of the proposed methods has the same complexity as an iteration of the Policy Gradient (PG) methods. 

Although the CMDP problem is nonconcave (CMDP problem is typically a maximization problem) in policy $\pi$ (nonconcavity inherited from 
MDP problem, which is nonconcave even in the bandit case \cite{mei2020global}), the complexity $\tilde{\mathcal{O}}\left(1/\epsilon\right)$ fits lower bound for smooth concave problems with large number of constraints \cite{nemirovsky1992information,ouyang2021lower}. 
Despite that fact, if we have only a few constraints $m$ --- that is typical for most of the practical applications --- these results are not optimal and we may expect $\tilde{\mathcal{O}}\left(m\right)$ iteration complexity for concave problems with $m$ constraints \cite{gasnikov2016universal,gladin2020solving,Gladin2021SolvingSM,xu2020first}, which corresponds to lower bound for small enough $m$ \cite{nemirovsky1979problem}. In this paper we are transferring $\tilde{\mathcal{O}}\left(m\right)$ iteration complexity result to the nonconcave CMDP problem.

\subsection{Related work}
There is a considerable interest in RL / MDP problems \cite{sutton1999policy,puterman2014markov,bertsekas2019reinforcement} and CMDP problems \cite{altman1999constrained}. For the past ten years there was a great theoretical progress in different directions. For example,
given a generative model with $|\Sc|$ states and $|\Ac|$ actions, we can find $\epsilon$-policy ($\epsilon$ is a quality in terms of cumulative reward) for $\gamma$-discounted infinite-horizon MDP problem with
\begin{equation}\label{complexity}
    \tilde{\mathcal{O}}\left(\frac{|\Sc| \cdot |\Ac|}{(1-\gamma)^{3}\epsilon^{2}}\right)
\end{equation}
samples  \cite{sidford2018near,wainwright2019variance,agarwal2020model} (analogously for CMDP, see arXiv version of  \cite{jin2020efficiently}) that corresponds (up to logarithmic factors) to the lower bound from \cite{azar2012sample}. Moreover, the dependence on $\epsilon$ can be improved to $\log\left(1/\epsilon\right)$ at the expense of dependence on $|\Sc|$.
Unfortunately, in many practical applications these optimal algorithms do not work at all due to the size of $|\Sc|$.  

A popular way to escape from the curse of dimensionality is to use PG methods \cite{mnih2015human,schulman2015trust}, where a parameterized (for example by Deep Neural Networks \cite{li2017deep}) class of policies is considered.  In the core of PG-type methods for MDP problems lie gradient-type methods (Mirror Descent \cite{lan2022policy,zhan2021policy}, Natural Policy Gradient (NPG) \cite{kakade2001natural,cen2021fast}, etc.) in the space of parameters applied to a properly regularized (in proper proximal setup) cumulative reward maximization problem. The gradient is calculated by using policy gradient theorem \cite{sutton1999policy}, which reduces gradient calculation to $Q$-function \eg{(value function $V$)} estimation. Under proper choice of regularizers (proximal setups), these methods require $\tilde{\mathcal{O}}\left((1-\gamma)^{-1}\right)$ iterations (they converge linearly in function value and in policy) and are not sensitive to inexactness $\delta$ of $Q$-value estimation ($\delta \sim \epsilon$), see details in \cite{cen2021fast,lan2022policy,zhan2021policy} and reference therein. Given a generative model, it is possible to obtain from these results (see \cite{azar2012sample,agarwal2020model}) analogs of formula \eqref{complexity} for sample complexity that would be worse in terms of $(1-\gamma)$ dependence \cite{cen2021fast}, but can be better in terms of $|\Sc|$.

For CMDP problems, PG methods are also well-developed, see, e.g., surveys in \cite{lanarcpo,liu2021fast}
and references therein.  
The best (in terms of PG iterations) known complexity bounds were obtained in these works \cite{lanarcpo,liu2021fast,ying2022dual}.  

In \cite{ying2022dual} with additional strong assumption (initial state distribution covers the entire state space) the complexity bound $\tilde{\mathcal{O}}\left(\epsilon^{-1}\right)$  was obtained for entropy-regularized CMDP (for true CMDP -- $\tilde{\mathcal{O}}\left(\epsilon^{-2}\right)$).  In \cite{lanarcpo} the complexity bound $\tilde{\mathcal{O}}\left(\epsilon^{-1}\right)$ was obtained under weaker additional assumption (Markov chain induced by any stationary policy is ergodic) by using dual approach, see Section~\ref{contr} for the details. For both of these approaches, given a generative model, one can obtain analogs of formula \eqref{complexity} for sample complexity that would be worse not only in terms of $(1-\gamma)$ dependence but also in terms of $\epsilon$ (but still can be better in terms of $|\Sc|$).

In \cite{liu2021fast} the complexity bound $\tilde{\mathcal{O}}\left(\epsilon^{-1}\right)$ was obtained without additional assumptions. 
\eg{We summarize the described results in the Table~\ref{table:1}.}

\begin{table*} 
\centering
    \begin{tabular}{ |c|c|c|c|c| } 
     \hline
     Method & \thead{$\beta$-Uniform Ergodicity\\ assumption} & \# NPG method iterations & \thead{Accuracy\\ of $V^{\pi}$} & Samples\\ 
     \hline
     \hline
     AR-CPO \cite{lanarcpo} & \ding{51} & $\tilde{\mathcal{O}}\left(\frac{R_{\max}\sqrt{m}}{((1 - \beta)\xi)^{1/2}(1 - \gamma)^3\epsilon}\right)$ &
     $\sim\epsilon^2$
     & $\sim \epsilon^{-5}$
     \\ 
     \hline PMD-PD
     \cite{liu2021fast} & \ding{55} & $\tilde{\mathcal{O}}\left(\frac{1}{(1 - \gamma)^3\epsilon}\left[\frac{m}{(1 - \gamma)^2} + \frac{1}{\xi}\right]\right)$ &   $\sim\epsilon$ 
      & $\sim \epsilon^{-3}$
      \\ 
     \hline \textbf{This work}
      & \ding{51}  & 
     $\tilde{\mathcal{O}}\left(\frac{m}{(1-\gamma)\zeta} \right)$ &
     $\sim\epsilon^2$
      & $\sim \epsilon^{-4}$
    \\ 
     \hline
    \end{tabular}
\caption{
\eg{
Complexities of state-of-the-art methods for CMDP.  Notation: $\epsilon$ -- accuracy in terms of optimality gap and constraint violation, $\xi$ -- Slater's parameter (could be small), $\zeta$ -- numerical constant (small $\zeta \simeq 10^{-7}$, usually much larger in practice $\zeta \simeq 10^{-1}$) in Vaidya's cutting-plane method, $\beta$ -- mixing time parameter (could be close to 1), $R_{\max} \le r_{\max}\sqrt{m}$, $r_{\max}$ -- parameter that bounds all the rewards.\protect\footnotemark}
}
\label{table:1}
\end{table*}

\footnotetext{\eg{Samples estimate was obtained based on the results from \cite{cen2021fast,zhan2021policy}, where it was shown that $\varepsilon$ (accuracy of $V^{\pi}$ -- output of NPG), determines the accuracy $\delta$ of $Q$-function evaluation $\delta \sim (1 - \gamma)^2\varepsilon$. The sample complexity of $\delta$-value of $Q$-function is $|S||A|(1-\gamma)^{-3}\delta^{-2}$. So sample complexity is $\text{[\# NPG method iterations]}\cdot |S||A|(1-\gamma)^{-7}\varepsilon^{-2}$. Note that in such a way we obtain upper bound on samples complexity. In practice such theoretical bounds turn out to be greatly overestimated (they are far from being optimal in terms of $(1 - \gamma)$), rather than estimates for the number of NPG iterations.}}

\subsection{Main contributions}\label{contr}
In the core of our approach lies the paper \cite{lanarcpo}, where the authors introduce entropy-regularized policy optimizer and solve regularized dual problem by proper version of Nesterov's accelerated gradient method. 
First of all, they use the 
strong duality for CMDP problem, which can be derived \cite{paternain2019constrained} from the fact of compactness and convexity of the set of occupation measures  \cite{borkar1988convex} or
from Linear Programming representation of CMDP problem in discounted state-action visitation distribution \cite{altman1999constrained}. The next important step is entropy policy regularization. This regularization simultaneously solves several tasks at once. First of all, it allows to estimate the gradient of dual function using NPG method that has a linear rate of convergence (in policy) and is robust to inexactness in $Q$-function evaluations \cite{cen2021fast}. The linear rate of convergence in policy is crucial since the dual accelerated method is sensitive to inexactness in gradient, which can be controlled if we have convergence in policy. 
Secondly, this regularization allows to prove smoothness (in the spirit of \cite{nesterov2005smooth} and with additional nice analysis of Mitrophanov's perturbation bounds \eg{\cite{mitrophanov2005sensitivity,zou2019finite}} for showing that visitation measure is Lipschitz w.r.t. the policy) of the dual problem. 
The smoothness of the dual problem allows to use Nesterov's accelerated method to solve it and to get an optimal rate.
The last step is the regularization of the dual problem to obtain a linear rate of convergence for the dual accelerated method, which negates the fact that we should solve the dual problem with higher accuracy to obtain the desired accuracy for the primal problem and constraint violation \cite{devolder2012double,gasnikov2016efficient}. An alternative approach, which has not been realized yet, is related to the primal-dual analysis of the method, which is used for the dual problem, see \cite{nemirovski2010accuracy} for convex problems. 
In this approach, it is sufficient to solve the dual problem with the same accuracy as we wish to solve the primal one. This primal-dual approach may conserve the dependence on $\epsilon^{-2}$ in \eqref{complexity} in the final sample complexity estimate if the method we used for the dual problem does not accumulate an error in gradient over iterations. From \cite{nemirovski2010accuracy,gladin2020solving} it is known that the Ellipsoid method is a primal-dual one and does not accumulate an error in gradient. 

Our contribution \eg{is to replace the dual accelerated method in the approach described above} with Vaidya's cutting-plane method \cite{vaidya1989new,vaidya1996new}. Vaidya's method has a linear rate of convergence (without any regularization) and outperforms the accelerated method when dealing with small dimension problems \cite{bubeck_book}. Moreover, Vaidya's method does not accumulate an error in gradient value \cite{Gladin2021SolvingSM} and hence is more robust than the accelerated method.

We build a new way for CMDP problems to estimate the quality of the primal solution from the dual one. To the best of our knowledge, the developed technique is also new for standard convex (concave) inequalities constrained problems and quite different from the technique that was used in \cite{lanarcpo}. This technique can be applied to any linear convergent algorithms for the dual problem.
Moreover, we improve $|\Ac|$-times the bound on the Lipschitz gradient constant of the dual function from \cite{lanarcpo}. In Lemma 7 \cite{lanarcpo} the authors formulate the correct result, but in fact, they prove a $|\Ac|$-times worse result than it was formulated. We give accurate proof of Lemma 7 by using the result from Appendix of \cite{juditsky2005recursive}. 

\eg{Another important point is that our proposed method allows to obtain the final policy in a straightforward way by performing a call to the same policy optimizer which is used on every iteration.
By contrast, the algorithm from \cite{lanarcpo} can obtain the final policy formally only in the finite setting, by calculating $\nu_{\rho}^{\pi}$ in $\abs{\Sc} \abs{\Ac}$ calls to value oracle (the notation is introduced below), which loses its motivation with transition to continuous state spaces.}

Similarly to \cite{lanarcpo}, our proposed method
can be applied to a wider class of nonconvex/nonconcave constrained problems with strong duality (zero duality gap) and uniqueness of the solution of the auxiliary problem, which relates the primal variables with the dual ones.

Finally, we demonstrate by numerical experiments that our proposed algorithm indeed outperforms AR-CPO from \cite{lanarcpo} when $m$ is not too big. 

\section{Preliminaries}\label{sec_prelim}
\subsection{Markov Decision Process}
A Markov decision process (MDP) is determined by a five-tuple $(\mathcal{S}, \mathcal{A}, \mathrm{P}, r, \gamma)$, where $\mathcal{S}$ is the state space, $\mathcal{A}$ is the action space, $\mathrm{P}$ is the transition kernel, $r$ is the reward function and $\gamma \in(0,1)$ is the discount factor. Assume that $\mathcal{S}$ and $\mathcal{A}$ are finite with cardinality $|\mathcal{S}|$ and $|\mathcal{A}|$, respectively. The initial state $s_{0}$ follows a distribution $\rho$. At any time $t \in \mathbb{N}_{+}$, an agent takes an action $a_{t} \in \mathcal{A}$ at state $s_{t} \in \mathcal{S}$, after which, according to the distribution $\mathrm{P}\left(s_{t} \mid s_{t-1}, a_{t-1}\right)$, the environment transits to the next state and the agent receives a reward $r\left(s_{t}, a_{t}\right)$.  The goal is to maximize the expected accumulated discounted reward: $$\mathbb{E}\left[\sum_{t=0}^{\infty} \gamma^{t} r\left(s_{t}, a_{t}\right)\right].$$ 

A stationary policy maps a state $s \in \mathcal{S}$ to a distribution $\pi(\cdot \mid s)$ over $\mathcal{A}$, which does not depend on time $t$. For a given policy $\pi$, its value function for any initial state $s \in \mathcal{S}$  is defined as
\begin{align*}
V_{r}^{\pi}(s)&:=\mathbb{E}\left[\sum_{t=0}^{\infty} \gamma^{t} r\left(s_{t}, a_{t}\right) \mid s_{0}= s, a_{t} \sim \pi\left(a_{t} \mid s_{t}\right), \right. 
\\ & \qquad\qquad\qquad\qquad\qquad 
  s_{t+1} \sim \mathrm{P}\left(\cdot \mid s_{t}, a_{t}\right)\Big].
\end{align*}

Next, the mathematical expectation is taken with respect to the distribution of the initial state, the expected reward is determined by following policy $\pi$ as 
$$V_{r}^{\pi}(\rho):=\mathbb{E}_{s_{0} \sim \rho}\left[V_{r}^{\pi}\left(s_{0}\right)\right].$$
The discounted state-action visitation distribution defined $\nu_{\rho}^{\pi}$ as follows: 
\begin{align*}
\nu_{\rho}^{\pi}(s, a)&:=(1-\gamma) \sum_{t=0}^{\infty} \gamma^{t} \operatorname{Pr}\Big\{s_{t}=s, a_{t}=a \mid s_{0} \sim \rho, 
\\
& \qquad\qquad\quad
a_{t} \sim \pi\left(\cdot \mid s_{t}\right), s_{t+1} \sim \mathrm{P}\left(\cdot \mid s_{t}, a_{t}\right)\Big\},
\end{align*}
for any $s \in \mathcal{S}, a \in \mathcal{A}$. The value function thus can be equivalently written as
$$
V_{r}^{\pi}(\rho)=\frac{\sum_{s \in \mathcal{S}, a \in \mathcal{A}} \nu_{\rho}^{\pi}(s, a) r(s, a)}{1-\gamma}=\frac{\left\langle\nu_{\rho}^{\pi}, r\right\rangle_{\mathcal{S} \times \mathcal{A}}}{1-\gamma},
$$
where $\langle\cdot, \cdot\rangle_{\mathcal{S}} \times \mathcal{A}$ denotes the inner product over the space $\mathcal{S} \times \mathcal{A}$ by reshaping $\nu_{\rho}^{\pi}$ and $r$ as $|\mathcal{S}| \times|\mathcal{A}|$-dimensional vectors, and we omit the subscripts and use $\langle\cdot, \cdot\rangle$ when there is no confusion.

\subsection{Constrained MDP}
\label{CMDP_setting}

The difference between CMDP and MDP is that the reward is an $(m+1)$-dimensional vector: 
$$r(s, a)= \left[r_{0}(s, a), r_{1}(s, a), \ldots, r_{m}(s, a)\right]^{\top}.$$
Each reward function $r_{i},\ i=0,1, \ldots, m$ is positive and finite, $$r_{i, \max }:=\max _{s \in \mathcal{S}, a \in \mathcal{A}}\left\{r_{i}(s, a)\right\},$$ for $i=0,1, \ldots, m$; $R_{\max }:=\sqrt{\sum_{i=1}^{m} r_{i, \max }^{2}} .$

Then the value function defined with respect to the $i$-th component of the reward vector $r$ as follows 
\begin{align*}
V_{i}^{\pi}(s)
:= \mathbb{E} \Big[\sum_{t=0}^{\infty} \gamma^{t} r_{i}\left(s_{t}, a_{t}\right)\, \big|\, & s_{0}=s, a_{t} \sim \pi\left(a_{t} \mid s_{t}\right), 
\\ & s_{t+1} \sim \mathrm{P}\left(\cdot \mid s_{t}, a_{t}\right)\Big]
\end{align*}
and $V_{i}^{\pi}(\rho) = \mathbb{E}_{s_{0} \sim \rho}\left[V_{i}^{\pi} \left(s_{0} \right)\right],\ $ for $i=0,1, \ldots, m$. The objective of the constrained MDP is to solve the following constrained optimization problem:
\begin{equation}
\begin{aligned}
    \max_{\pi \in \Pi}\, & V_{0}^{\pi}(\rho)\\
    \text{s.t. } & V_{i}^{\pi}(\rho) &\geq c_{i}, \quad i=1, \ldots, m,
    \label{optprob}
\end{aligned}
\end{equation}
$\Pi=\left\{\pi \in \mathbb{R}^{|\mathcal{S}||\mathcal{A}|}: \sum_{a \in \mathcal{A}} \pi(a \mid s)=1,\ \pi(a \mid s) \geq 0\right.$, $\forall(s, a) \in \mathcal{S} \times \mathcal{A}\}$ is the set of all stationary policies.\\
Let $\pi^{*}$ denote the optimal policy for the problem \eqref{optprob}. The goal is to 
find an $\epsilon$-optimal policy defined as follows. 
\begin{definition}
A policy $\tilde{\pi}$ is $\epsilon$-optimal if its corresponding optimality gap and the constraint violation satisfy
$$
V_{0}^{*}(\rho)-V_{0}^{\tilde{\pi}}(\rho) \leq \epsilon ; \text { and }\left\|\left[c-V^{\tilde{\pi}}(\rho)\right]_{+}\right\|_{2} \leq \epsilon,
$$
where $V_{0}^{*}(\rho)$ is the optimal value of \eqref{optprob}, $$V^{\tilde{\pi}}(\rho):=\left[V_{1}^{\tilde{\pi}}(\rho), \ldots, V_{m}^{\tilde{\pi}}(\rho) \right]^{\top}$$
and $c:=\left[c_{1}, \ldots, c_{m}\right]^{\top}.$
\end{definition}

\subsection{Notation}
Let $I_m$ denote $m\times m$ identity matrix, $\mathbf{1}_m \in \mathbb{R}^m$ be a vector of ones. The set of nonnegative real numbers is denoted by $\R_+$. Notation $\operatorname{int} P$ is used for the interior of a set $P \subseteq \R^m$. Given a vector $x \in \R^m$, let $\| x \|_p, p \in [1, \infty]$ denote the $p$-norm of $x$, $[x]_+$ be defined by $\left([x]_+\right)_i=\max\{0, x_i\}, i=1, \ldots, m$. For two vectors $x,y \in \R^m$, inner product is denoted by $\langle x, y\rangle$ or $x^\top y$.
Given two functions $f(\epsilon)$ and $g(\epsilon)$, we write $f(\epsilon)=\mathcal{O}(g(\epsilon))$ if there exists some constant $C>0$, such that $f(\epsilon) \leq C g(\epsilon)$ for small enough $\epsilon$. $\tilde{\mathcal{O}}(g(\epsilon))(\cdot)$ means $\mathcal{O}(\cdot)$ up to logarithmic factor in a small power (usually 1 or 2). Bold symbol $\pmb{\pi}$ denotes the constant ($\pmb{\pi} = 3.1415\ldots$) contrary to plain $\pi$ which indicates policy. $\log(\cdot)$ is the natural logarithm.

\section{Cutting-plane algorithm for CMDP}
In this section, we introduce the Cutting-plane algorithm for CMDP which is presented in Algorithm \ref{alg:vaidya}.
The algorithm assumes access to:

\begin{enumerate}
    \item Oracles, sufficent to run 
    Natural policy gradient algorithm (see Appendix \ref{NPGs}),
    for our MDP with arbitrary rewards.
    For example, access to exact gradient of value function
w.r.t. softmax policy parametrization for any vector
of rewards, and exact Fisher information matrix for a
softmax-parametrized policy. Or, in our finite setting, also access
to soft Q-functions is enough.
    \item Exact value function of a policy w.r.t
    to constraints' reward vectors $r_1, \dots, r_m$. 
\end{enumerate}

In the algorithm, policies may be stored as full 
softmax parametrizations (\ref{softmax_parametrization}),
or directly as  vectors
$\pi(a\vert s), \,\,\forall (s,a) \in \Sc \times \Ac$
(both ways are equivalent and can be calculated one from another).

The core idea is to consider the entropy-regularized Lagrange function
\begin{equation*}
    \L_{\tau}(\pi, \lambda) := V_{0}^{\pi}(\rho)+\left\langle\lambda, V^{\pi}(\rho)-c\right\rangle + \tau \H(\pi),
\end{equation*}
where $\lambda \in \R^m_+$ is the vector of dual variables,
\begin{equation}\label{constraints_notation}
    V^{\pi}(\rho)=\left[V_{1}^{\pi}(\rho), \ldots, V_{m}^{\pi}(\rho)\right]^{\top}\; \text{ and }\; c=\left[c_{1}, \ldots, c_{m}\right]^{\top}
\end{equation}
are respectively vectors of constraints and constraint thresholds,
\begin{multline*}
    \H(\pi)=-\mathbb{E}
    \bigg[\sum_{t=0}^{\infty} \gamma^{t} \log \left(\pi\left(a_{t} \mid s_{t}\right)\right) \Bigm| \\
    s_{0}\sim \rho, a_{t} \sim \pi\left(a_{t} \mid s_{t}\right), s_{t+1} \sim \mathrm{P}\left(\cdot \mid s_{t}, a_{t}\right)\bigg]
\end{multline*}
is the discounted entropy of the policy $\pi$ and $\tau>0$ is the regularization coefficient. The proposed method is based on two components: Vaidya's cutting-plane method \cite{vaidya1989new, vaidya1996new} for solving the dual problem
\begin{equation}\label{dual_probl}
    \min_{\lambda \in \R^m_+} \bigl\{ d_{\tau}(\lambda) := \max_{\pi \in \Pi} \L_{\tau}(\pi, \lambda) \bigr\},
\end{equation}
and an entropy-regularized policy optimizer for solving the inner problem $\max_{\pi \in \Pi} \L_{\tau}(\pi, \lambda)$ on each iteration of the outer loop.
Below is the description of both components.



\subsection{Entropy-regularized policy optimizer}
To estimate the gradient of the dual function $\nabla d_{\tau}(\lambda)$, one has to solve the problem $\max_{\pi \in \Pi} \L_{\tau}(\pi, \lambda)$. Note that this is equivalent to maximizing an entropy-regularized value function corresponding to a reward function $r_{\lambda}:=r_{0}+\sum_{i=1}^{m} \lambda_{i} r_{i}$. As mentioned in the introduction, entropy regularization enables the linear rate of convergence of an NPG method. In 
what follows, 
$\mathrm{NPG}\left(r_{\lambda}, \tau, \delta\right)$ represents 
a call to NPG procedure that learns the policy $\tilde{\pi}_{\tau, \lambda}$ for the entropy-regularized MDP with the regularization coefficient $\tau$ and reward $r_{\lambda}$ up to $\delta$-accuracy in terms of $l_{\infty}$ distance to the unique optimal regularized policy, i.e.,
\begin{equation}\label{npg_accuracy}
    \eg{\normbig{\pi_{\tau, \lambda}^* - \tilde{\pi}_{\tau, \lambda}}_\infty < \delta,}
\end{equation}
\eg{where $\pi_{\tau, \lambda}^*:= \argmax_{\pi \in \Pi} \L_{\tau}(\pi, \lambda)$ (existence and uniqueness of such optimal policy is addressed below).}
More details on entropy-regularized policy optimizers are provided in Appendix \ref{NPGs}.

\begin{algorithm*}[t] 
	\caption{Cutting-plane algorithm for CMDP}
	\label{alg:vaidya}
	\begin{algorithmic}[1]
		\REQUIRE number of outer iterations $T$,
		NPG accuracy $\delta$ (see \eqref{npg_accuracy}),
		pair $(A_0, b_0) \in \mathbb{R}^{k_0\times m} \times \mathbb{R}^{k_0}$,
		algorithm parameters $\eta \leq 10^{-4}$,
		$\zeta \leq 10^{-3} \cdot \eta$.
		\FOR{$t=0,\, \dots, \, T-1$}
		    \STATE $\lambda_t:=\operatorname{VolCenter}(A, b)$
		    \STATE Compute
    		    $H_t^{-1} := \left( H(\lambda_t; A_t,b_t) \right)^{-1}$ and
    		    $\displaystyle \left\{ \sigma_{i}(\lambda_t; A_t,b_t) \right\}_{i=1}^{k_t}$ 
		    \STATE $\displaystyle i_t := \argmin_{1 \leq i \leq k_t} \sigma_{i}(\lambda_t; A_t,b_t)$
		    \IF {$\sigma_{i_t}(\lambda_t; A_t,b_t) < \zeta$}
		        \STATE Obtain $\left(A_{t+1}, b_{t+1}\right)$ by removing the $i_t$-th row from $\left(A_t, b_t\right)$,
		        \STATE $k_{t+1} := k_t - 1.$
		    \ELSE
		        \IF {$\lambda_t \in \mathbb{R}_+^{\eg{m}}$}
    		        \STATE $\pi_t := \mathrm{NPG}\left(r_0 + \langle \lambda_t, r \rangle, \tau, \delta \right)$ \label{line_NPG}
    		        \STATE $\widehat{\nabla}_{t} := c - V^{\pi_t}(\rho)\; $ \eg{ (see notation in \eqref{constraints_notation}),} \label{line_antigrad} 
		        \ELSE
    		        \STATE Define $\widehat{\nabla}_{t}$ as the vector with components\label{separ_oracle_vmdp}
    		        $$ (\widehat{\nabla}_{t})_{i}= \begin{cases}1, & (\lambda_t)_{i}<0, \\ 0, & (\lambda_t)_{i} \geq 0,\end{cases}\quad i=1,\ldots, m. $$
		        \ENDIF
		        \STATE Find such $\beta_t \in \mathbb{R}$ that $\widehat{\nabla}_{t}^\top \lambda_t \geq \beta_t$ from the equation
		        $$\frac{\widehat{\nabla}_{t}^\top H_t^{-1} \widehat{\nabla}_{t}}{(\widehat{\nabla}_{t}^\top \lambda_t - \beta_t)^2} = \frac{1}{2} \sqrt{\eta \zeta},$$
		        \STATE $A_{t+1} := \begin{pmatrix}A_t\\ \widehat{\nabla}_{t}^{\top}\end{pmatrix},\;\;b_{t+1} := \begin{pmatrix}b_t\\\beta_t\end{pmatrix},\;\;k_{t+1} = k_t + 1$.
		    \ENDIF
		\ENDFOR
		\STATE $\lambda_T = \argmin\limits_{\lambda \in \{\lambda_0, ..., \lambda_{T-1}\}} d_{\tau}(\lambda)$
		\STATE $\pi_T := \mathrm{NPG}\left(r_0 + \langle \lambda_T, r \rangle, \tau, \delta \right)$\label{final_policy}
		\ENSURE $\pi_T$.
	\end{algorithmic}
\end{algorithm*}

\subsection{Vaidya's cutting-plane method}
Vaidya's cutting-plane method \cite{vaidya1989new, vaidya1996new} in an algorithm for a convex optimization problem with complexity $\O(m \log \frac{m}{\epsilon})$, which makes it a good choice for formulations with a small or moderate dimensionality like the dual problem \eqref{dual_probl}. Moreover, it has been shown that the method can be used with an inexact subgradient, and it does not accumulate the error \cite{Gladin2021SolvingSM}. This makes it very suitable for the problem \eqref{dual_probl}, since the gradient of the dual function is computed approximately.

We will now introduce necessary formulas and present the proposed method for the problem \eqref{optprob}.
For a matrix $A \in \mathbb{R}^{k \times m}$ with rows $a_{i}^{\top}, i=1, \ldots, k$, and a vector $b \in \mathbb{R}^k$, define
\begin{align}
    P(A,b) &:= \{\lambda \in \mathbb{R}^m: \, A\lambda \geq b\}, \\
    \label{hess}
    H(\lambda; A,b) &:= \sum_{i=1}^{k} \frac{a_{i} a_{i}^{\top}}{\left(a_{i}^{\top} \lambda-b_{i}\right)^{2}},\\
    \label{vaidya_sigmas}
    \sigma_{i}(\lambda; A,b) &:= \frac{a_{i}^{\top} \left(H(\lambda; A,b)\right)^{-1} a_{i}}{\left(a_{i}^{\top} \lambda-b_{i}\right)^{2}}, \quad 1 \leq i \leq k,
\end{align}
\begin{multline}
    \label{vol_center}
    \operatorname{VolCenter}(A, b) := \\ \argmin_{\lambda \in \operatorname{int} P(A, b)} \big\{ \eg{\mathcal{V}}(\lambda; A,b) := \frac{1}{2} \log \left(\operatorname{det} H(\lambda; A,b)\right) \big\},
\end{multline}
where $\operatorname{det}H(\lambda; A,b)$ denotes the determinant of $H(\lambda; A,b)$. Since $\eg{\mathcal{V}}$ is a self-concordant function of $x$, it can be efficiently minimized with Newton-type methods. The algorithm starts with a pair $(A_0, b_0) \in \mathbb{R}^{k_0\times m} \times \mathbb{R}^{k_0}$, such that $P(A_0, b_0)$ contains the search space. We refer the reader to Appendix \ref{vaidya_descr} for more information on the original Vaidya's method and its parameters.

\section{Convergence results for the proposed algorithm}\label{sec_converg}
First, we introduce technical assumptions on our CMDP instance
$(\Sc, \Ac, P, \gamma, r_0, r_1, \ldots, r_m, \rho)$,
that are widely used in reinforcement learning literature.
\begin{assumption}[Slater Condition]\label{assumption_slater}
    There exists a constant $\xi \in \mathbb{R}_{+}$, and at least one policy $\pi_{\xi} \in \Pi$, such that for all $i=1, \ldots, m,\, V_{i}^{\pi_{\xi}}(\rho) \geq c_{i}+\xi$.
\end{assumption}
Slater condition asserts that there exists a strictly feasible policy.
Define the set
\begin{multline}
    \Lambda := \{ \lambda \in \R^m_+\; |\; \| \lambda \|_1 \leq B_{\lambda} \}\; \text{ with } \\
    B_{\lambda}:=\frac{r_{0, \max } + \log |\Ac|}{(1-\gamma) \xi}.
\end{multline}

\begin{assumption}[Regularized optimal policy uniqueness]\label{assumption_tau_optimal_unique}
    For any $\tau > 0, \lambda \in \Lambda$ there exists exactly one optimal policy for the problem:
    \begin{equation}
        \max_{\pi \in \Pi} \L_{\tau}(\pi, \lambda),
    \end{equation}
    which we call $\pi_{\tau,\lambda}^*$.
\end{assumption}

\begin{remark}
As was proved in \cite{UniqueRegularized}, for any tabular MDP
there is a unique policy that maximizes 
$V^{\pi}_{\tau}(s)$ for all $s \in S$ at once, given that $\tau > 0$.
This policy obviously will be optimal with any initial distribution
$\rho$, but we assume
that it will still be unique. We use recurrent equations
on it that were proved in \cite{UniqueRegularized} as well.
\end{remark}

\begin{assumption}[Uniform Ergodicity]\label{assumption_ergodicity}
    For any $\lambda \in \Lambda$, the Markov chain induced by the policy $\pi_{\tau, \lambda}^{*}$ and the Markov transition kernel is uniformly ergodic, i.e., there exist constants $C_{M}>0$ and $0<\beta<1$ such that for all $t \geq 0$,
    $$
    \sup _{s \in \mathcal{S}} d_{T V}\left(\mathbb{P}\left(s_{t} \in \cdot \mid s_{0}=s\right), \chi_{\pi_{\tau, \lambda}^{*}}\right) \leq C_{M} \beta^{t},
    $$
    where $\chi_{\pi_{\tau, \lambda}^{*}}$ is the stationary distribution of the MDP induced by policy $\pi_{\tau, \lambda}^{*}$, and $d_{T V}(\cdot, \cdot)$ is the total variation distance.
\end{assumption}
Convergence rate of Algorithm \ref{alg:vaidya} in terms of the optimality gap $V_{0}^{*}(\rho)-V_{0}^{\pi_T}(\rho)$ and the constraint violation $\left[c-V^{\pi_T}(\rho)\right]_{+}$ is described by the following theorem. The proof can be found in Appendix \ref{main_proof}.

\begin{theorem}\label{convergence_theorem}
Suppose Assumptions \ref{assumption_slater}, \ref{assumption_tau_optimal_unique} and \ref{assumption_ergodicity}
hold, let $T\in \mathbb{N}$ be fixed and $\epsilon$ denote the value
\begin{equation}\label{dual_unopt}
    \epsilon := \frac{2m^2 B_{\lambda}}{\zeta} \left(\xi + \frac{\sqrt{m} R_{max}}{1-\gamma} \right) \exp \left( \frac{\log \pmb{\pi} - \zeta T}{2m} \right),
\end{equation}
where $\pmb{\pi}$ denotes the constant ($\pmb{\pi} = 3.1415\ldots$).
The Cutting-plane algorithm for CMDP (Algorithm~\ref{alg:vaidya}) with parameters
\begin{gather}
    \notag \eg{k_0 := m+1,}\quad
    A_0 := \left[\begin{array}{c}
        -I_m \\
        1
    \end{array}\right],\\
    \quad b_0 := \left[\begin{array}{c}
        B_{\lambda} \mathbf{1}_m \\
        m B_{\lambda}
    \end{array}\right],\quad
    \tau := \eg{\min\{1, \sqrt[3]{\epsilon}\}},
\end{gather}
\eg{number of outer iterations T and NPG accuracy $\delta>0$ (see \eqref{npg_accuracy})}
provides the following convergence guarantee of the optimality gap and the constraint violation:
\begin{align}
    \label{opt_gap}
    V_{0}^{*}(\rho) - V_{0}^{\pi_T}(\rho) &\leq \frac{B_{\lambda} R_{max} \sqrt{2 m L_{\beta}}}{1-\gamma} \sqrt{\epsilon^{2/3}+6\gamma\delta} \\
    &+2\epsilon \nonumber +18\gamma \delta \sqrt[3]{\epsilon} + \frac{\log |\Ac|}{1-\gamma} \sqrt[3]{\epsilon} \\
    &+ \sqrt{m} B_{\lambda}
    \eg{\frac{L_{\beta}\abs{\Ac}R_{\max}}
    {(1-\gamma)} \delta},\\
    \label{constr_viol}
    \notag \bigl\| [ c- V^{\pi_T}(\rho) ]_+ \bigr\|_2 &\leq \frac{2 R_{max}^2 L_{\beta}}{1-\gamma} (\epsilon^{2/3}+6\gamma\delta) \\
    &+\eg{\frac{L_{\beta}\abs{\Ac}R_{\max}}
    {(1-\gamma)} \delta},
\end{align}
where  $$L_{\beta}:=\left\lceil\log _{\beta}\left(C_{M}^{-1}\right)\right\rceil+(1-\beta)^{-1}+1.$$
\end{theorem}

The value $\epsilon$ in \eqref{dual_unopt} reflects linear convergence of Algorithm~\ref{alg:vaidya} in terms of the dual function. As it can be seen from \eqref{opt_gap} and \eqref{constr_viol}, this also implies linear convergence in terms of value function and constraint violation, if NPG provides appropriate accuracy. Thus, the algorithm results in the following complexity bound
\begin{corollary}\label{main_coroll}
Algorithm \ref{alg:vaidya} outputs an $\epsilon$-optimal policy with respect to both the optimality gap and constraint violation after
\begin{equation}\label{asympt}
    T = \mathcal{O}\left( \frac{m}{\zeta} \log \frac{m\eg{\log\abs{\Ac}}}{
    \eg{(1-\beta)(1-\gamma)\zeta\xi}\epsilon} \right)
\end{equation}
steps. The total number of calls to the policy gradient
oracle made in all NPG calls is:
\begin{multline}\label{asympt_npg}
    \notag N_{oracle} = \mathcal{O}\left(
    T \cdot \eg{\frac{1}{1-\gamma}\log \frac{m\log\abs{\Ac}}{(1-\beta)(1-\gamma)\xi\epsilon}}
    \right)=\\
    \notag = 
    \mathcal{O}\bigg(
  \frac{m}{\eg{(1-\gamma)\zeta}} \cdot \log \frac{m\log\abs{\Ac}}
  {(1-\beta)(1-\gamma)\xi\epsilon}\cdot\\
  \cdot\log \frac{m\log\abs{\Ac}}{(1-\beta)(1-\gamma)\zeta\xi\epsilon}
    \bigg).
\end{multline}
\end{corollary}
Proof of the corollary is in Appendix \ref{coroll_proof}.

\begin{remark}\eg{
From the proof of the Corollary 1 in Appendix \cite{lanarcpo},  the number of policy gradient method iterations is $$N_{\text{AR-CPO}} = \frac{R_{\max}\sqrt{m}}{((1 - \beta)\xi)^{1/2}(1 - \gamma)^3\epsilon},$$ where we skip not only constants, but also $\log$-factors, which could be close to zero, $R_{\max}^2 \le r_{\max}^2 m$.
Note that for the concurrent paper \cite{liu2021fast} 
$$N_{\text{PMD-PD}} = \frac{1}{(1 - \gamma)\epsilon}\left[\frac{m}{(1 - \gamma)^2} + \frac{1}{\xi}\right]$$
up to $\log$-factors. For our approach, $N_{\text{Vaidya}} = \frac{m}{(1 - \gamma)\zeta}$ up to $\log$-factors (depending on  $\xi$, $\beta$, $\gamma$, $\epsilon$), $\zeta$ – small numerical parameter of Vaidya’s algorithm. 
Thus, from these formulas we may conclude that up to a $\log$-factor, our approach is theoretically better when $\epsilon \lesssim \zeta\cdot (1 - \gamma)^{-2}$.
But in reality this $\log$-factor might be significant.}
\end{remark}

\subsection{Regularization of dual variables}
The proposed approach can also be modified in the following way: the dual problem writes as
\begin{equation}
    \max_{\lambda \in \R^m_+} d_{\tau,\mu}(\lambda):=d_{\tau}(\lambda) + \frac{\mu}{2} \| \lambda \|_2^2,
\end{equation}
where $\mu>0$ is the regularization coefficient. In this case, the vector $\widehat{\nabla}_{t}$ in Line~\ref{line_antigrad} of Algorithm~\ref{alg:vaidya} will be replaced with 
$$\widehat{\nabla}_{t} := c - V^{\pi_t}(\rho) - \mu \lambda_t.$$ If $\mu$ is chosen sufficiently small, the
result of Corollary \ref{main_coroll} 
remains true, see Appendix \ref{dual_reg_append} for details.

\begin{figure}[th!]
\centering
     
    \begin{subfigure}[t]{0.48\textwidth}
         \includegraphics[width =  1\textwidth ]{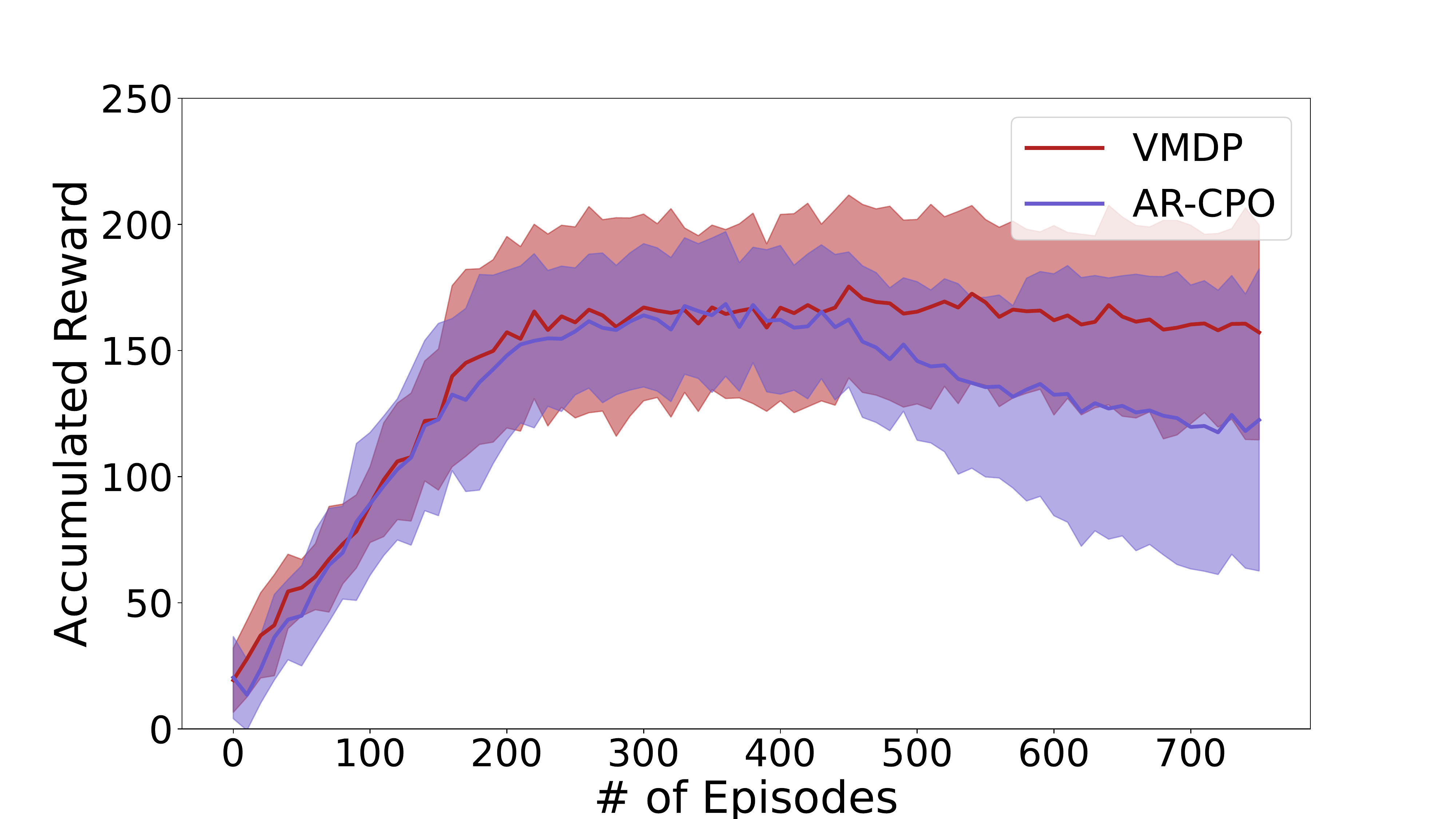}
         \caption{}
         \label{fig:1a}
    \end{subfigure}
    \begin{subfigure}[t]{0.48\textwidth}
         \centering
        \includegraphics[width =   \textwidth]{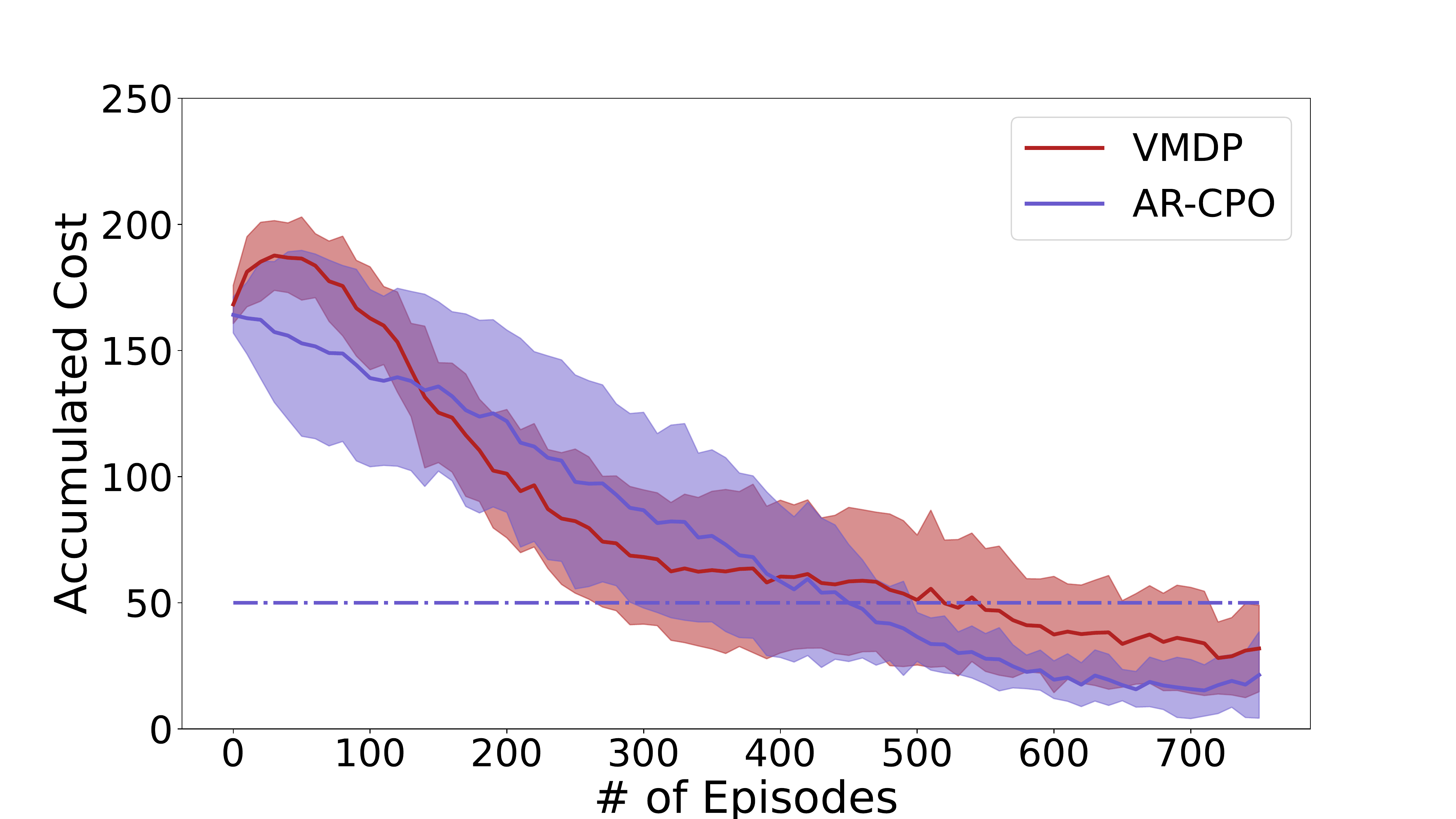}
        \caption{} \label{fig:1b}
    \end{subfigure}
    \begin{subfigure}[t]{0.48\textwidth}
         \centering
        \includegraphics[width =   \textwidth]{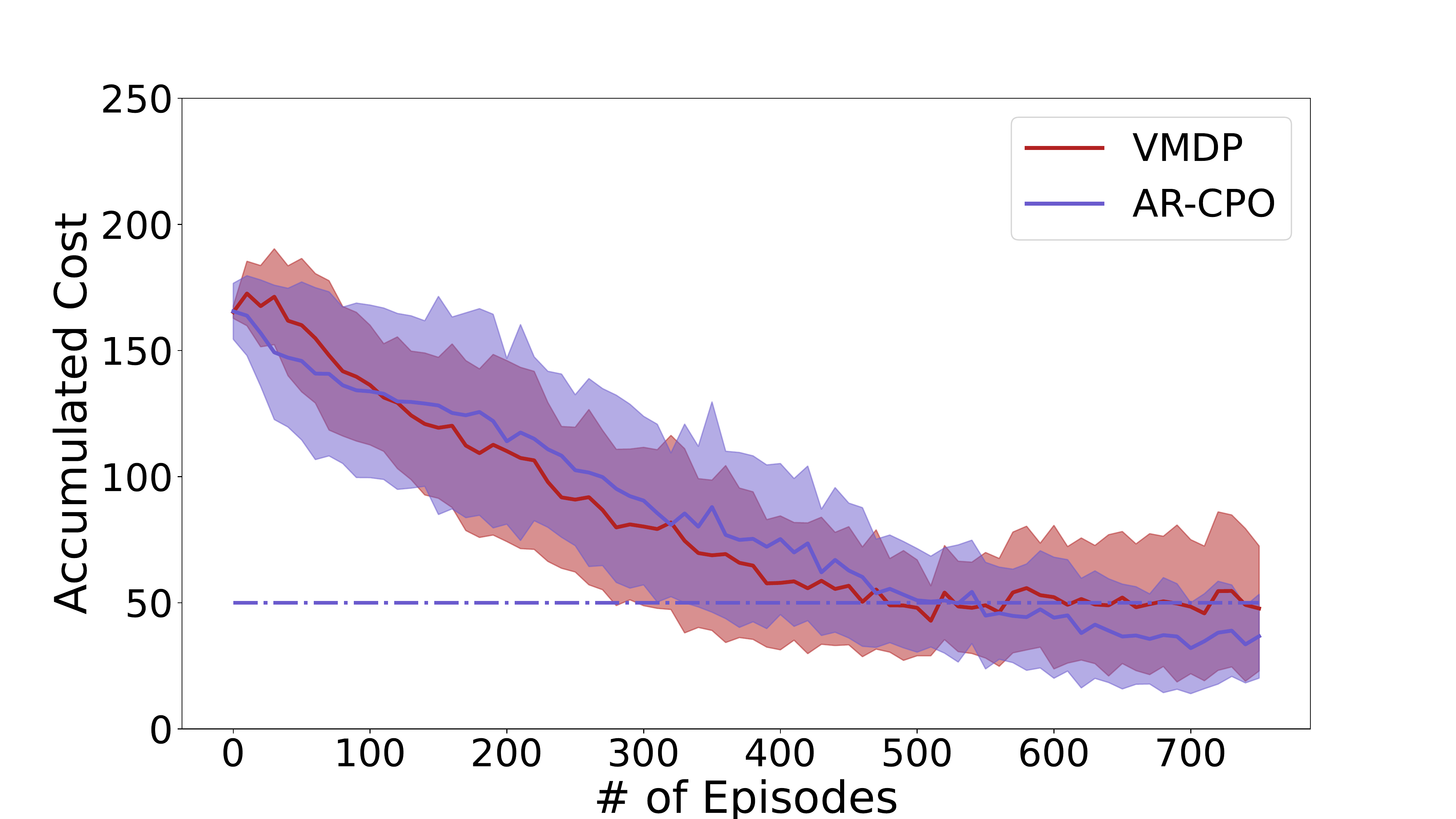}
        \caption{} \label{fig:1c}
    \end{subfigure}
    \caption{ Average performance for VMDP and AR-CPO; the $x$-axis is training iteration.}
    \label{fig:perf1}
\end{figure}

\section{Experiments}\label{sec:main_experiments}

For our experiments we used Acrobot-v1, OpenAI Gym \cite{Mei2020OnTG} environment. This environment contains two links connected linearly to form a chain, with one end of the chain fixed. The joint between the two links is actuated. The goal is to swing the end of the lower link up to a given height. Two additional constraints are implemented in order to have similar environment as in \cite{lanarcpo} for comparison purpose.

In Figure \ref{fig:perf1} we compare our cutting-plane algorithm (VMDP) with the state-of-the-art primal-dual optimization (AR-CPO) method for CMDP \cite{lanarcpo}. In order for a fair comparison, the same neural softmax policy and the trust region policy \cite{pmlr-v37-schulman15} optimization are used in both algorithms.

Similarly to \cite{lanarcpo} we picture the average over 10 random initialized seeds and translucent error bands have the width of two standard deviations. The hyper parameters of AR-CPO algorithm are optimal from \cite{lanarcpo}. More information about experiments and parameters settings can be found in Appendix 
\ref{sec:paramA} and \ref{sec:paramB}.


\begin{figure}
\centering
    \begin{subfigure}[t]{0.48\textwidth}
         \includegraphics[width =  1.0\textwidth ]{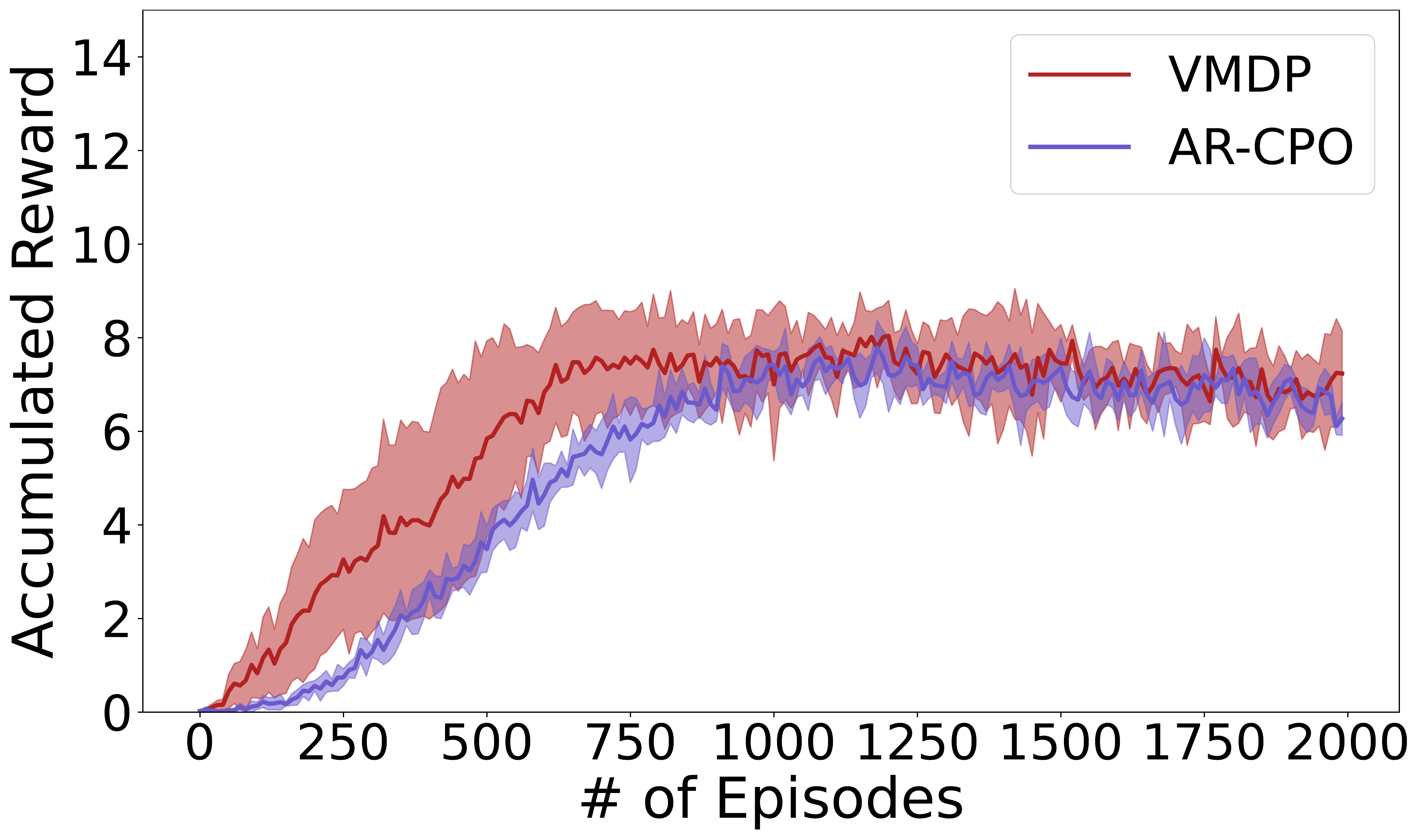}
         \caption{}
         \label{fig:suppl_2a}
    \end{subfigure}
    \begin{subfigure}[t]{0.48\textwidth}
         \centering
        \includegraphics[width =  1.0\textwidth]{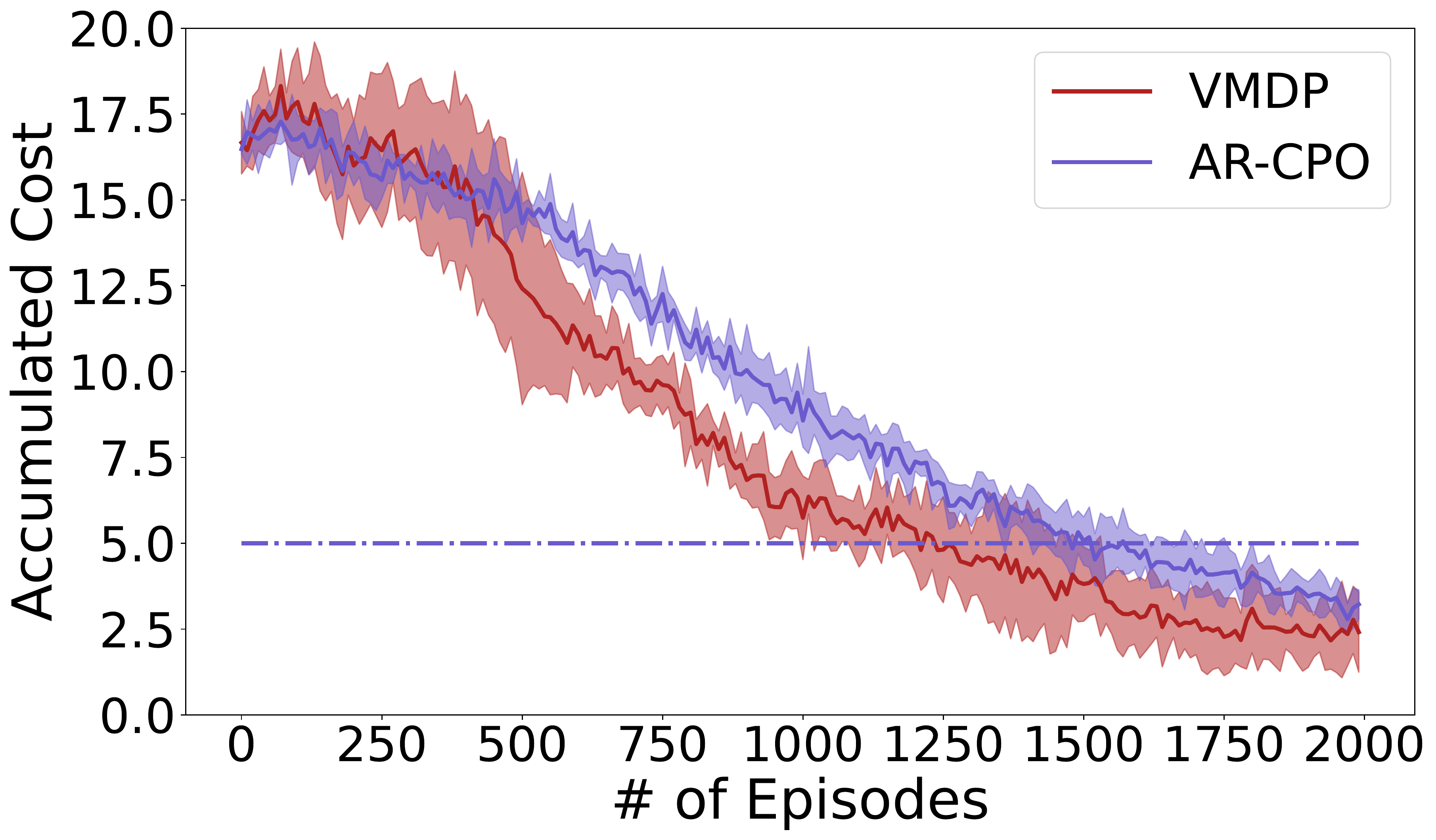}
        \caption{} \label{fig:suppl_2b}
    \end{subfigure}
    \begin{subfigure}[t]{0.48\textwidth}
         \centering
        \includegraphics[width =  1.0\textwidth]{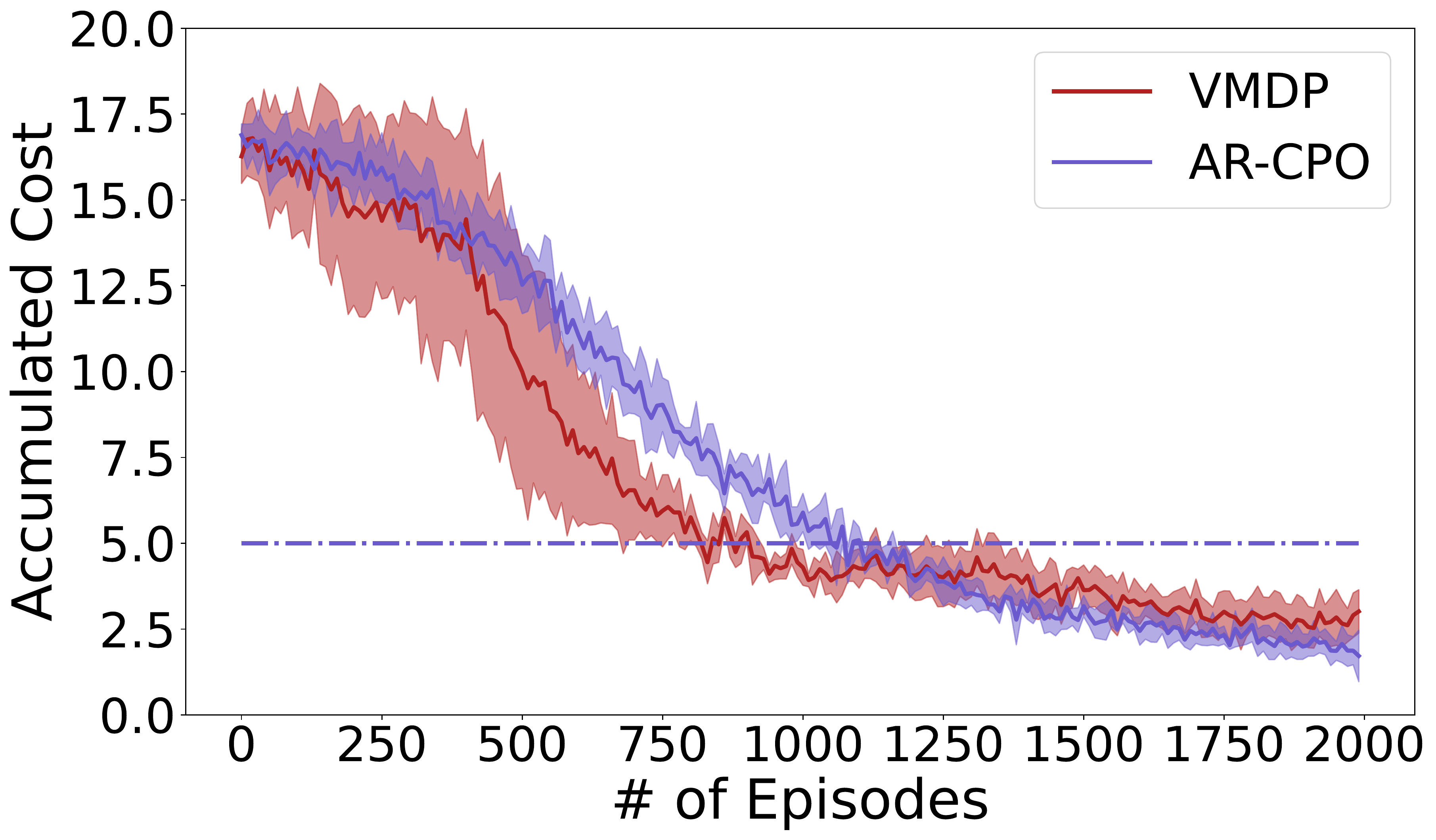}
        \caption{} \label{fig:suppl_2c}
    \end{subfigure}
    \caption{ Average performance of VMDP and AR-CPO in case of discounted reward and costs; the $x$-axis is training iteration.}
    \label{fig:perf_suppl}
\end{figure}

Figure \ref{fig:1a} represents average total reward over episode, while Figures \ref{fig:1b} and \ref{fig:1c} show constraints with dashed line as the constraint thresholds. We used total reward for a fair comparison with existing state-of-the-art approach. Moreover, in some cases total reward is more important in practise. 

We find that our algorithm achieves higher total reward with similar standard deviation. The speed of converge of both algorithms is similar. Thus, our algorithm allows to achieve better performance with the same training time for MDP tasks with the small number of constraints.

\paragraph{Discounted reward experiment.}
Previously we considered experiments, where we calculated total reward and costs. We will now briefly describe the results of the experiments with discounted reward and costs. 
Considering similar environment as before, we set thresholds of 5 to the discounted costs. Figure \ref{fig:perf_suppl} provides the comparison between AR-CPO with VMDP under their best tuned parameters provided in Appendix \ref{sec:additional_expparam}.

In our experiment we used large number of policy optimization steps in subroutine equal to 40 for VMDP. This allowed our algorithm to solve NPG subroutine with high accuracy and, as a result, to converge faster. In the same time, increasing the number of steps in subroutine did not make AR-CPO converge faster. It showed the best performance with this parameter equal to 1.

From Figure \ref{fig:suppl_2a}, we observe that VMDP converges faster than AR-CPO, which is consistent with theory. 

Thus, VMDP algorithm is useful in both discounted and total reward cases and shows better performance than AR-CPO.

\section{Conclusion}
In this paper we consider the constrained Markov decision process, where an agent aims to maximize the expected accumulated discounted reward subject to a relatively small number $m$ of constraints on its costs. The best known algorithms achieve $\tilde{\mathcal{O}}\left(1/\epsilon\right)$ iteration complexity to find global optimum, where $\epsilon$ characterizes optimality gap and constraint violation. Each iteration of these algorithms has the same complexity as an iteration of the Policy Gradient (PG) methods. In this paper we improve (for relatively small number $m$) iteration complexity bound and obtain linear convergence $\tilde{\mathcal{O}}\left(m\right)$.


\subsubsection*{Acknowledgments}
The work of E. Gladin is funded by the Deutsche Forschungsgemeinschaft (DFG, German Research Foundation) under Germany's Excellence Strategy – The Berlin Mathematics Research Center MATH+ (EXC-2046/1, project ID: 390685689).

The work of A. Gasnikov was supported by a grant for research centers in the field of
artificial intelligence, provided by the Analytical Center for the
Government of the Russian Federation in accordance with the subsidy
agreement (agreement identifier 000000D730321P5Q0002 ) and the agreement
with the Ivannikov Institute for System Programming of the Russian
Academy of Sciences dated November 2, 2021 No. 70-2021-00142.

\clearpage

\bibliographystyle{abbrvnat}
\bibliography{references}

\begin{thebibliography}{49}
\providecommand{\natexlab}[1]{#1}
\providecommand{\url}[1]{\texttt{#1}}
\expandafter\ifx\csname urlstyle\endcsname\relax
  \providecommand{\doi}[1]{doi: #1}\else
  \providecommand{\doi}{doi: \begingroup \urlstyle{rm}\Url}\fi

\bibitem[Agarwal et~al.(2020)Agarwal, Kakade, and Yang]{agarwal2020model}
A.~Agarwal, S.~Kakade, and L.~F. Yang.
\newblock Model-based reinforcement learning with a generative model is minimax
  optimal.
\newblock In \emph{Conference on Learning Theory}, pages 67--83. PMLR, 2020.

\bibitem[Altman(1999)]{altman1999constrained}
E.~Altman.
\newblock \emph{Constrained Markov decision processes: stochastic modeling}.
\newblock Routledge, 1999.

\bibitem[Azar et~al.(2012)Azar, Munos, and Kappen]{azar2012sample}
M.~G. Azar, R.~Munos, and B.~Kappen.
\newblock On the sample complexity of reinforcement learning with a generative
  model.
\newblock \emph{arXiv preprint arXiv:1206.6461}, 2012.

\bibitem[Bertsekas(1991)]{DanskinTheorem}
Bertsekas.
\newblock \emph{Nonlinear programming}.
\newblock Athena Scientific, 1991.

\bibitem[Bertsekas(2019)]{bertsekas2019reinforcement}
D.~Bertsekas.
\newblock \emph{Reinforcement learning and optimal control}.
\newblock Athena Scientific, 2019.

\bibitem[Borkar(1988)]{borkar1988convex}
V.~S. Borkar.
\newblock A convex analytic approach to markov decision processes.
\newblock \emph{Probability Theory and Related Fields}, 78\penalty0
  (4):\penalty0 583--602, 1988.

\bibitem[Bubeck(2015)]{bubeck_book}
S.~Bubeck.
\newblock Convex optimization: Algorithms and complexity.
\newblock \emph{Found. Trends Mach. Learn.}, 8\penalty0 (3–4):\penalty0
  231–357, nov 2015.
\newblock ISSN 1935-8237.
\newblock \doi{10.1561/2200000050}.
\newblock URL \url{https://doi.org/10.1561/2200000050}.

\bibitem[Cen et~al.(2021)Cen, Cheng, Chen, Wei, and Chi]{cen2021fast}
S.~Cen, C.~Cheng, Y.~Chen, Y.~Wei, and Y.~Chi.
\newblock Fast global convergence of natural policy gradient methods with
  entropy regularization.
\newblock \emph{Operations Research}, 2021.

\bibitem[Devolder et~al.(2012)Devolder, Glineur, and
  Nesterov]{devolder2012double}
O.~Devolder, F.~Glineur, and Y.~Nesterov.
\newblock Double smoothing technique for large-scale linearly constrained
  convex optimization.
\newblock \emph{SIAM Journal on Optimization}, 22\penalty0 (2):\penalty0
  702--727, 2012.

\bibitem[et. al.(2021)]{MDPrate}
S.~C. et. al.
\newblock Fast global convergence of natural policy gradient methodswith
  entropy regularization.
\newblock \emph{arXiv preprint arXiv:2007.06558}, 2021.

\bibitem[Fisac et~al.(2018)Fisac, Akametalu, Zeilinger, Kaynama, Gillula, and
  Tomlin]{fisac2018general}
J.~F. Fisac, A.~K. Akametalu, M.~N. Zeilinger, S.~Kaynama, J.~Gillula, and
  C.~J. Tomlin.
\newblock A general safety framework for learning-based control in uncertain
  robotic systems.
\newblock \emph{IEEE Transactions on Automatic Control}, 64\penalty0
  (7):\penalty0 2737--2752, 2018.

\bibitem[Garcia and Fernandez(2015)]{garcia2015comprehensive}
J.~Garcia and F.~Fernandez.
\newblock A comprehensive survey on safe reinforcement learning.
\newblock \emph{Journal of Machine Learning Research}, 16\penalty0
  (1):\penalty0 1437--1480, 2015.

\bibitem[Gasnikov et~al.(2016{\natexlab{a}})Gasnikov, Kamzolov, and
  Mendel]{gasnikov2016universal}
A.~Gasnikov, D.~Kamzolov, and M.~Mendel.
\newblock Universal composite prox-method for strictly convex optimization
  problems.
\newblock \emph{arXiv preprint arXiv:1603.07701}, 2016{\natexlab{a}}.

\bibitem[Gasnikov et~al.(2016{\natexlab{b}})Gasnikov, Gasnikova, Nesterov, and
  Chernov]{gasnikov2016efficient}
A.~V. Gasnikov, E.~Gasnikova, Y.~E. Nesterov, and A.~Chernov.
\newblock Efficient numerical methods for entropy-linear programming problems.
\newblock \emph{Computational Mathematics and Mathematical Physics},
  56\penalty0 (4):\penalty0 514--524, 2016{\natexlab{b}}.

\bibitem[Gladin et~al.(2020)Gladin, Kuruzov, Stonyakin, Pasechnyuk, Alkousa,
  and Gasnikov]{gladin2020solving}
E.~Gladin, I.~Kuruzov, F.~Stonyakin, D.~Pasechnyuk, M.~Alkousa, and
  A.~Gasnikov.
\newblock Solving strongly convex-concave composite saddle point problems with
  a small dimension of one of the variables, 2020.

\bibitem[Gladin et~al.(2021)Gladin, Sadiev, Gasnikov, Dvurechensky, Beznosikov,
  and Alkousa]{Gladin2021SolvingSM}
E.~Gladin, A.~Sadiev, A.~V. Gasnikov, P.~E. Dvurechensky, A.~Beznosikov, and
  M.~S. Alkousa.
\newblock Solving smooth min-min and min-max problems by mixed oracle
  algorithms.
\newblock \emph{Communications in Computer and Information Science}, 2021.

\bibitem[Jin and Sidford(2020)]{jin2020efficiently}
Y.~Jin and A.~Sidford.
\newblock Efficiently solving mdps with stochastic mirror descent.
\newblock In \emph{International Conference on Machine Learning}, pages
  4890--4900. PMLR, 2020.

\bibitem[Juditsky et~al.(2005)Juditsky, Nazin, Tsybakov, and
  Vayatis]{juditsky2005recursive}
A.~B. Juditsky, A.~V. Nazin, A.~B. Tsybakov, and N.~Vayatis.
\newblock Recursive aggregation of estimators by the mirror descent algorithm
  with averaging.
\newblock \emph{Problems of Information Transmission}, 41\penalty0
  (4):\penalty0 368--384, 2005.

\bibitem[Kakade(2001)]{kakade2001natural}
S.~M. Kakade.
\newblock A natural policy gradient.
\newblock \emph{Advances in neural information processing systems}, 14, 2001.

\bibitem[Lan(2022)]{lan2022policy}
G.~Lan.
\newblock Policy mirror descent for reinforcement learning: Linear convergence,
  new sampling complexity, and generalized problem classes.
\newblock \emph{Mathematical programming}, pages 1--48, 2022.

\bibitem[Li et~al.(2016)Li, Jiang, Li, Zheng, and You]{li2016cmdp}
P.~Li, Y.~Jiang, W.~Li, F.~Zheng, and X.~You.
\newblock A cmdp-based approach for energy efficient power allocation in
  massive mimo systems.
\newblock In \emph{2016 IEEE Wireless Communications and Networking
  Conference}, pages 1--6. IEEE, 2016.

\bibitem[Li et~al.(2021)Li, Guan, Zou, Xu, Liang, and Lan]{lanarcpo}
T.~Li, Z.~Guan, S.~Zou, T.~Xu, Y.~Liang, and G.~Lan.
\newblock Faster algorithm and sharper analysis for constrained markov decision
  process, 2021.
\newblock URL \url{https://arxiv.org/abs/2110.10351}.

\bibitem[Li(2017)]{li2017deep}
Y.~Li.
\newblock Deep reinforcement learning: An overview.
\newblock \emph{arXiv preprint arXiv:1701.07274}, 2017.

\bibitem[Liu et~al.(2021)Liu, Zhou, Kalathil, Kumar, and Tian]{liu2021fast}
T.~Liu, R.~Zhou, D.~Kalathil, P.~Kumar, and C.~Tian.
\newblock Fast global convergence of policy optimization for constrained mdps.
\newblock \emph{arXiv preprint arXiv:2111.00552}, 2021.

\bibitem[Mei et~al.(2020{\natexlab{a}})Mei, Xiao, Szepesvari, and
  Schuurmans]{Mei2020OnTG}
J.~Mei, C.~Xiao, C.~Szepesvari, and D.~Schuurmans.
\newblock On the global convergence rates of softmax policy gradient methods.
\newblock In \emph{ICML}, 2020{\natexlab{a}}.

\bibitem[Mei et~al.(2020{\natexlab{b}})Mei, Xiao, Szepesvari, and
  Schuurmans]{mei2020global}
J.~Mei, C.~Xiao, C.~Szepesvari, and D.~Schuurmans.
\newblock On the global convergence rates of softmax policy gradient methods.
\newblock In \emph{International Conference on Machine Learning}, pages
  6820--6829. PMLR, 2020{\natexlab{b}}.

\bibitem[Mitrophanov(2005)]{mitrophanov2005sensitivity}
A.~Y. Mitrophanov.
\newblock Sensitivity and convergence of uniformly ergodic markov chains.
\newblock \emph{Journal of Applied Probability}, 42\penalty0 (4):\penalty0
  1003--1014, 2005.

\bibitem[Mnih et~al.(2015)Mnih, Kavukcuoglu, Silver, Rusu, Veness, Bellemare,
  Graves, Riedmiller, Fidjeland, Ostrovski, et~al.]{mnih2015human}
V.~Mnih, K.~Kavukcuoglu, D.~Silver, A.~A. Rusu, J.~Veness, M.~G. Bellemare,
  A.~Graves, M.~Riedmiller, A.~K. Fidjeland, G.~Ostrovski, et~al.
\newblock Human-level control through deep reinforcement learning.
\newblock \emph{nature}, 518\penalty0 (7540):\penalty0 529--533, 2015.

\bibitem[Nachum et~al.(2017)Nachum, Norouzi, Xu, and
  Schuurmans]{UniqueRegularized}
O.~Nachum, M.~Norouzi, K.~Xu, and D.~Schuurmans.
\newblock Bridging the gap between value and policy based reinforcement
  learning.
\newblock \emph{arXiv preprint arXiv:1702.08892}, 2017.

\bibitem[Nemirovski et~al.(2010)Nemirovski, Onn, and
  Rothblum]{nemirovski2010accuracy}
A.~Nemirovski, S.~Onn, and U.~G. Rothblum.
\newblock Accuracy certificates for computational problems with convex
  structure.
\newblock \emph{Mathematics of Operations Research}, 35\penalty0 (1):\penalty0
  52--78, 2010.

\bibitem[Nemirovsky and Yudin(1979)]{nemirovsky1979problem}
A.~Nemirovsky and D.~Yudin.
\newblock Problem complexity and optimization method efficiency.
\newblock \emph{M.: Nauka (in Russian)}, 1979.

\bibitem[Nemirovsky(1992)]{nemirovsky1992information}
A.~S. Nemirovsky.
\newblock Information-based complexity of linear operator equations.
\newblock \emph{Journal of Complexity}, 8\penalty0 (2):\penalty0 153--175,
  1992.

\bibitem[Nesterov(2005)]{nesterov2005smooth}
Y.~Nesterov.
\newblock Smooth minimization of non-smooth functions.
\newblock \emph{Mathematical programming}, 103\penalty0 (1):\penalty0 127--152,
  2005.

\bibitem[Ono et~al.(2015)Ono, Pavone, Kuwata, and Balaram]{ono2015chance}
M.~Ono, M.~Pavone, Y.~Kuwata, and J.~Balaram.
\newblock Chance-constrained dynamic programming with application to risk-aware
  robotic space exploration.
\newblock \emph{Autonomous Robots}, 39\penalty0 (4):\penalty0 555--571, 2015.

\bibitem[Ouyang and Xu(2021)]{ouyang2021lower}
Y.~Ouyang and Y.~Xu.
\newblock Lower complexity bounds of first-order methods for convex-concave
  bilinear saddle-point problems.
\newblock \emph{Mathematical Programming}, 185\penalty0 (1):\penalty0 1--35,
  2021.

\bibitem[Paternain et~al.(2019)Paternain, Chamon, Calvo-Fullana, and
  Ribeiro]{paternain2019constrained}
S.~Paternain, L.~Chamon, M.~Calvo-Fullana, and A.~Ribeiro.
\newblock Constrained reinforcement learning has zero duality gap.
\newblock \emph{Advances in Neural Information Processing Systems}, 32, 2019.

\bibitem[Polyak(1987)]{polyak1983intro}
B.~T. Polyak.
\newblock Introduction to optimization.
\newblock \emph{Inc., Publications Division, New York}, 1987.

\bibitem[Puterman(2014)]{puterman2014markov}
M.~L. Puterman.
\newblock \emph{Markov decision processes: discrete stochastic dynamic
  programming}.
\newblock John Wiley \& Sons, 2014.

\bibitem[Schulman et~al.(2015{\natexlab{a}})Schulman, Levine, Abbeel, Jordan,
  and Moritz]{pmlr-v37-schulman15}
J.~Schulman, S.~Levine, P.~Abbeel, M.~Jordan, and P.~Moritz.
\newblock Trust region policy optimization.
\newblock 37:\penalty0 1889--1897, 07--09 Jul 2015{\natexlab{a}}.
\newblock URL \url{https://proceedings.mlr.press/v37/schulman15.html}.

\bibitem[Schulman et~al.(2015{\natexlab{b}})Schulman, Levine, Abbeel, Jordan,
  and Moritz]{schulman2015trust}
J.~Schulman, S.~Levine, P.~Abbeel, M.~Jordan, and P.~Moritz.
\newblock Trust region policy optimization.
\newblock In \emph{International conference on machine learning}, pages
  1889--1897. PMLR, 2015{\natexlab{b}}.

\bibitem[Sidford et~al.(2018)Sidford, Wang, Wu, Yang, and Ye]{sidford2018near}
A.~Sidford, M.~Wang, X.~Wu, L.~Yang, and Y.~Ye.
\newblock Near-optimal time and sample complexities for solving markov decision
  processes with a generative model.
\newblock \emph{Advances in Neural Information Processing Systems}, 31, 2018.

\bibitem[Sutton et~al.(1999)Sutton, McAllester, Singh, and
  Mansour]{sutton1999policy}
R.~S. Sutton, D.~McAllester, S.~Singh, and Y.~Mansour.
\newblock Policy gradient methods for reinforcement learning with function
  approximation.
\newblock \emph{Advances in neural information processing systems}, 12, 1999.

\bibitem[Vaidya(1989)]{vaidya1989new}
P.~M. Vaidya.
\newblock A new algorithm for minimizing convex functions over convex sets.
\newblock In \emph{30th Annual Symposium on Foundations of Computer Science},
  pages 338--343. IEEE Computer Society, 1989.

\bibitem[Vaidya(1996)]{vaidya1996new}
P.~M. Vaidya.
\newblock A new algorithm for minimizing convex functions over convex sets.
\newblock \emph{Mathematical programming}, 73\penalty0 (3):\penalty0 291--341,
  1996.

\bibitem[Wainwright(2019)]{wainwright2019variance}
M.~J. Wainwright.
\newblock Variance-reduced $ q $-learning is minimax optimal.
\newblock \emph{arXiv preprint arXiv:1906.04697}, 2019.

\bibitem[Xu(2020)]{xu2020first}
Y.~Xu.
\newblock First-order methods for problems with o (1) functional constraints
  can have almost the same convergence rate as for unconstrained problems.
\newblock \emph{arXiv preprint arXiv:2010.02282}, 2020.

\bibitem[Ying et~al.(2022)Ying, Ding, and Lavaei]{ying2022dual}
D.~Ying, Y.~Ding, and J.~Lavaei.
\newblock A dual approach to constrained markov decision processes with entropy
  regularization.
\newblock In \emph{International Conference on Artificial Intelligence and
  Statistics}, pages 1887--1909. PMLR, 2022.

\bibitem[Zhan et~al.(2021)Zhan, Cen, Huang, Chen, Lee, and Chi]{zhan2021policy}
W.~Zhan, S.~Cen, B.~Huang, Y.~Chen, J.~D. Lee, and Y.~Chi.
\newblock Policy mirror descent for regularized reinforcement learning: A
  generalized framework with linear convergence.
\newblock \emph{arXiv preprint arXiv:2105.11066}, 2021.

\bibitem[Zou et~al.(2019)Zou, Xu, and Liang]{zou2019finite}
S.~Zou, T.~Xu, and Y.~Liang.
\newblock Finite-sample analysis for sarsa with linear function approximation.
\newblock \emph{Advances in neural information processing systems}, 32, 2019.

\end{thebibliography}

\newpage


\onecolumn 
\appendix

\part*{Supplementary Materials}

\section{Natural policy gradient (NPG)}\label{NPGs}

NPG is one of the algorithms that can efficently optimize
a finite MDP with relative entropy regularization:
\begin{align}
    \max_{\pi \in \Pi} V^{\pi}_{\tau}(\rho) =
    V^{\pi}(\rho) + \tau\Hc(\pi),
\end{align}
assuming access to gradients of their soft value function
and to a Fisher information matrix
respective to its softmax parametrization. Specifically,
policies are parametrized as follows:
\begin{align}
\label{softmax_parametrization}
    &\pi^{\theta}(a \vert s) = \frac{\exp(e^{\theta_{s,a}})}
    {\sum_{a'} \exp(e^{\theta_{s,a'}})},\\
    &\theta \in \mathbb{R}^{\Sc \times \Ac},
\end{align}
and the NPG algorithm has access to functions:
\begin{align}
    &G(\theta) = \nabla_{\theta} V^{\pi^\theta}_{\tau}(\rho),\\
	&\mathcal{F}_\rho^\pi(\theta) := 
	\mathop\mathbb{E}\limits_{s\sim d_{\rho}^{\pi^{\theta}},
	a\sim \pi\theta(\cdot|s)}
	\left[
	{\big( \nabla_{\theta} \log \pi_\theta(a|s) \big) \big( \nabla_{\theta} \log \pi_\theta(a|s) \big)^\top}
	\right].
\end{align}

This type of oracle is motivated by a possibility of estimating 
this gradient in high-dimension MDP's in applications.

\subsection{Algorithm and convergence rates}

The algorithm looks like this:

\begin{algorithm}[h!]
    \caption{Natural policy gradient (NPG) algorithm}
    \label{NPGalgorithm}
    \begin{algorithmic}[1]
        \REQUIRE learning rate $\eta > 0$, initialization 
        parameters $\theta_0 = 0$.
        
        \FOR{$t=0,1,2,\ldots$}
        
        \STATE $
        \theta^{(t + 1)} ~\gets~
        \theta^{(t)}
        + \eta \big(\mathcal{F}_\rho^{\theta^{(t)}}\big)^\dagger \nabla_{\theta}
        V_{\tau}^{\pi_{\theta^{(t)}}} (\rho)    
        $
        
	    \ENDFOR
    \end{algorithmic}
\end{algorithm}

($M^{\dagger}$ is Moore-Penrose inverse function)

The update rule (20) can be rewritten in terms of policies and soft Q-functions:
\begin{align}\label{eq:entropy_npg}
	\forall (s,a)\in \Sc\times \Ac: \quad 
	\pi^{(t+1)}(a|s) = 
 	\frac{1} {Z^{(t)}(s)} \big( \pi^{(t)}(a|s) \big)^{1- \frac{\eta\tau}{1-\gamma}} 
		\exp\Big( \frac{\eta{Q_{\tau}^{(t)}(s, a)}} { {1-\gamma}} \Big),
\end{align}
where $Z^{(t)}(s)= \sum_{a'\in \Ac}\big( \pi^{(t)}(a'|s) \big)^{1- \frac{\eta\tau}{1-\gamma}} 
\exp\big( \frac{\eta{Q_{\tau}^{(t)}(s, a')}} { {1-\gamma}} \big)$
is a normalizing coefficient, and soft Q-functions are defined as follows:
\begin{align}
    \forall (s,a) \in \Sc\times \Ac:\qquad	Q_{\tau}^{\pi} (s,a) & = r(s,a) + \gamma \mathbb{E}_{s'\sim P(\cdot|s,a)} \big[ V^{\pi}_{\tau}(s') \big], 
\end{align}

So, in our finite setting we can assume we are given an oracle of soft Q-functions instead of the earlier mentioned.

In \cite{MDPrate} the following theorem is proved
(in setting with $r(s,a) \in [0, 1]$):

\begin{theorem}[Linear convergence of exact entropy-regularized NPG]
\label{thm:npg_exact}
For any learning rate $0 < \eta  \leq (1-\gamma)/\tau$, the entropy-regularized NPG updates \eqref{eq:entropy_npg} satisfy 
\begin{align}
&\normbig{Q_{\tau}^{*} - Q_{\tau}^{(t+1)}}_\infty \leq C_1 \gamma(1 - \eta\tau)^{t}, \\
&\normbig{\log \pi_{\tau}^* - \log\pi^{(t+1)}}_\infty \leq 2C_1 \tau^{-1}(1 - \eta \tau)^{t}, \\
&\normbig{V_{\tau}^{*} - V_{\tau}^{(t+1)}}_\infty \leq 3C_1 \gamma(1 - \eta\tau)^{t}.
\end{align}
for all $t\geq 0$, where
\begin{equation}\label{eq:C1}
C_1 := \normbig{Q_{\tau}^* - Q_{\tau}^{(0)}}_\infty + 2\tau  
\left(1 - \frac{\eta \tau}{1 - \gamma}\right)
\normbig{\log \pi_{\tau}^* - \log \pi^{(0)}}_\infty.
\end{equation}
\end{theorem} 

We will use the algorithm with $\eta = \frac{1-\gamma}{\tau}$.
In this case we have convergence rates:
\begin{align}
&\normbig{Q_{\tau}^{*} - Q_{\tau}^{(t+1)}}_\infty  \leq C_1 
\gamma^{t+1}, \\
&\normbig{\pi_{\tau}^* - \pi_{\tau}^{(t+1)}}_\infty \leq
\normbig{\log \pi_{\tau}^* - \log\pi^{(t+1)}}_\infty  \leq 2C_1 \tau^{-1}\gamma^{t}, \\
&\normbig{V_{\tau}^{*} - V_{\tau}^{(t+1)}}_\infty \leq 3C_1 \gamma^{t+1},\\
&\text{with  } 
C_1 =  \normbig{Q_{\tau}^* - Q_{\tau}^{(0)}}_\infty.
\end{align}

\subsection{Usage of NPG in our work}

In our algorithm we need to solve auxiliary
problems of the form:
\begin{align}
    \max_{\pi \in \Pi} V^{\pi}_{\tau, \lambda}(\rho)
    := V_0^\pi(\rho) 
    + \sum_{i=1}^m \lambda_i V_i^{\pi}(\rho)
    + \tau\Hc(\pi)
\end{align}
with this procedure.

Since the objective is a $\tau$-regularized value
function for an MDP with rewards $r_0 + \langle \lambda, r \rangle$,
we can use the NPG procedure to optimize it.
However, we cannot pass our MDP with these rewards
directly to this method,
because it assumes $r(s,a) \in [0, 1]$ in \cite{MDPrate}.
So, we will scale both $r$ and $\tau$ to make
rewards satisfy this condition, and run NPG 
with a higher accuracy.

Specifically, we define a procedure $\mathrm{NPG}(r, \tau, \delta)$ as follows.

First, define $R = \max(\max_{s,a} r(s,a), 1)$. Calculate $r'(s,a) = \frac{r(s,a)}{R}$.

Then, apply NPG algorithm to solve an MDP with rewards $r'$ and regularization
coefficient $\frac{\tau}{R}$ with accuracy $\frac{\delta}{R}$ in policy norm.
For that we need a number of iterations $t + 1$ that satisfies: 

\begin{align}
\label{npg_deltaReq}
    &2\frac{C_1}{R}\left(\frac{\tau}{R}\right)^{-1}\gamma^t < \frac{\delta}{R},\\
    &\gamma^t < \frac{\delta\tau}{2C_1R},\\
    &t > \frac{
    \log{2C_1R} + \log{\delta^{-1}} + \log{\tau^{-1}}
    }
    {\log \gamma^{-1}}.
\end{align}

By this we get a $\delta$-optimal policy in terms of $l_{\infty}$ distance 
to the optimal policy, since it is the same after rescaling and $R \geq 1$.
Also, by \ref{npg_deltaReq}:

\begin{align}
    &2C_1\tau^{-1}\gamma^t < \frac{\delta}{R},\\
    &3C_1\gamma^{t+1} < \frac{6\tau\gamma}{R}\delta,\\
    &\normbig{V_{\tau}^{*} - V_{\tau}^{(t+1)}}_\infty \leq 6\tau\gamma\delta.
\end{align}

Finally, we get this statement:

\begin{theorem}\label{npg_convergence}
Suppose $\delta, \tau > 0$, and we have a $\tau$-regularized MDP.
Let $R = \max(\max_{s,a} r(s, a), 1)$.
Then a number of NPG iterations more than:
\begin{align}
    T = \frac{
    \log{2C_1R} + \log{\delta^{-1}} + \log{\tau^{-1}}
    }
    {\log \gamma^{-1}} + 1
\end{align}
is enough for our procedure to acquire
a policy $\tilde{\pi}$ that satisfies:
\begin{align}
    &\normbig{\pi_{\tau}^* - \tilde{\pi}}_\infty < \delta,\\
    &\normbig{V_{\tau}^{*} - V_{\tau}^{\tilde{\pi}}}_\infty \leq 6\tau\gamma\delta.
\end{align}
\end{theorem}

\section{Description of Vaidya's cutting-plane method}\label{vaidya_descr}
Vaidya proposed a cutting-plane method from \cite{vaidya1989new, vaidya1996new} for solving problems of the form
\begin{equation}\label{problem_vaidya}
    \min_{\lambda \in \Lambda} d(\lambda),
\end{equation}
where $\Lambda$ is a compact convex set with non-empty interior, $d(\lambda)$ is a continuous convex function. We will now introduce the notation and describe the algorithm. Let $P(A,b)$ denote the bounded full-dimensional polytope of the form
$$
P(A,b) = \{\lambda \in \mathbb{R}^m: \, A\lambda\geq b\}\, \text{ where } A \in \mathbb{R}^{k\times m} \text{ and } b \in \mathbb{R}^k.
$$
The logarithmic barrier for $P$ is defined as
$$
L(\lambda; A,b) := -\sum_{i=1}^{k} \log \left(a_{i}^{\top} \lambda-b_{i}\right),
$$
where $a_{i}^{\top}$ is the $i^{th}$ row of $A,\ i=1,\ldots,k$. The Hessian of $L(\lambda)$ is given by
\begin{equation}\label{hess_app}
    H(\lambda; A,b) =\sum_{i=1}^{k} \frac{a_{i} a_{i}^{\top}}{\left(a_{i}^{\top} \lambda-b_{i}\right)^{2}},
\end{equation}
and is positive definite for all $\lambda$ in $\operatorname{int} P$ (interior of $P$). The \textit{volumetric barrier} for $P(A,b)$ is defined as
$$
\eg{\mathcal{V}}(\lambda; A,b) = \frac{1}{2} \log \left(\operatorname{det} H(\lambda; A,b)\right),
$$
where $\operatorname{det}H(\lambda; A,b)$ denotes the determinant of $H(\lambda; A,b)$.
Let also $\sigma_{i}(\lambda; A,b)$ denote the values
\begin{equation}\label{vaidya_sigmas_app}
    \sigma_{i}(\lambda; A,b)=\frac{a_{i}^{\top} \left(H(
\lambda; A,b)\right)^{-1} a_{i}}{\left(a_{i}^{\top} \lambda-b_{i}\right)^{2}}, \quad 1 \leq i \leq k.
\end{equation}
\textit{Volumetric center} of $P$ is defined as the point that minimizes $\eg{\mathcal{V}}(\lambda; A,b)$ over the interior of $P$:
\begin{equation}\label{vol_center2}
    \operatorname{VolCenter}(A, b) := \argmin_{\lambda \in \operatorname{int} P(A, b)} \eg{\mathcal{V}}(\lambda; A,b).
\end{equation}
Volumetric barrier $\eg{\mathcal{V}}$ is a self-concordant function and can therefore be efficiently minimized with the Newton-type methods. For more details and theoretical analysis, refer to \cite{vaidya1996new,vaidya1989new}.
Consider the following version of inexact subgradient. 
\begin{definition}\label{def_delta_subgrad}
Vector $\nu \in \mathbb{R}^m$ is called a $\delta$-subgradient of a convex function $f$ at $z \in \operatorname{dom}d$ (denoted $\nu \in \partial_{\delta}d(z)$), if $$d(\lambda) \geq d(z) + \nu^{\top}(\lambda - z) -\delta \quad \forall \lambda \in \operatorname{dom}d.$$
If $\delta=0$, this we get the usual definition of subgradient $\partial_ {\delta} d(z) = \partial d(z)$.
\end{definition}
It has been proved that one can use $\delta$-subgradient instead of the exact subgradient in Vaidya's method \cite{Gladin2021SolvingSM}. Algorithm \ref{alg:og_vaidya} gives the version of the method using $\delta$-subgradients. 
The method produces a sequence of pairs $\left(A_t, b_t\right) \in \mathbb{R}^{k_t\times m}\times \mathbb{R}^{k_t}$, such that the corresponding polytopes contain a solution of the problem \eqref{problem_vaidya}. A simplex containing the set $\Lambda$ is often taken as the initial polytope $\left(A_0, b_0\right)$. \eg{For example, if $\|\lambda\|_2\leq \mathcal{R}$ for any $\lambda \in \Lambda$, then a possible choice of a starting polytope is
$$ P_0=\Bigl\{\lambda \in \mathbb{R}^m: \lambda_{j} \geqslant-\mathcal{R}, j=\overline{1,m},\ \sum_{j=1}^{m} \lambda_{j} \leqslant m \mathcal{R} \Bigr\} \supseteq \mathcal{B}_{\mathcal{R}} \supseteq \Lambda,$$
that is,
$$
    k_0 = m+1,\quad b_0 = - \mathcal{R} \left[\begin{array}{c}
        \mathbf{1}_m \\
        m
    \end{array}\right],\quad A_0 = \left[\begin{array}{c}
        I_m \\ 
        -\mathbf{1}_m^{\top}
    \end{array}\right].
$$
}
    
\begin{algorithm}[h!]
	\caption{Vaidya's cutting-plane method with $\delta$-subgradient}
	\label{alg:og_vaidya}
	\begin{algorithmic}[1]
		\REQUIRE number of outer iterations $T$,
		pair $(A_0, b_0) \in \mathbb{R}^{k_0\times m} \times \mathbb{R}^{k_0}$,
		algorithm parameters $\eta \leq 10^{-4}$,
		$\zeta \leq 10^{-3} \cdot \eta$.
		\FOR{$t=0,\, \dots, \, T-1$}
		    \STATE $\lambda_t:=\operatorname{VolCenter}(A, b)$
		    \STATE Compute
    		    $H_t^{-1} := \left( H(\lambda_t; A_t,b_t) \right)^{-1}$ and
    		    $\displaystyle \left\{ \sigma_{i}(\lambda_t; A_t,b_t) \right\}_{i=1}^{k_t}$ 
		    \STATE $\displaystyle i_t := \argmin_{1 \leq i \leq k_t} \sigma_{i}(\lambda_t; A_t,b_t)$
		    \IF {$\sigma_{i_t}(\lambda_t; A_t,b_t) < \zeta$}
		        \STATE Obtain $\left(A_{t+1}, b_{t+1}\right)$ by removing the $i_t$-th row from $\left(A_t, b_t\right)$,
		        \STATE $k_{t+1} := k_t - 1.$
		    \ELSE
		        \IF {$\lambda_t \in \mathbb{R}_+^{\eg{m}}$}
    		       \STATE Take $\widehat{\nabla}_{t} \in -\partial_{\delta} d(\lambda_t)$, 
		        \ELSE
    		        \STATE Take $\widehat{\nabla}_{t}$ such that $\widehat{\nabla}_{t}^{\top} \lambda \geq \widehat{\nabla}_{t}^{\top} \lambda_t\ \forall \lambda \in \Lambda$.\label{line_separ_oracle_vaidya}
		        \ENDIF
		        \STATE Find such $\beta_t \in \mathbb{R}$ that $\widehat{\nabla}_{t}^\top \lambda_t \geq \beta_t$ from the equation
		        $$\frac{\widehat{\nabla}_{t}^\top H_t^{-1} \widehat{\nabla}_{t}}{(\widehat{\nabla}_{t}^\top \lambda_t - \beta_t)^2} = \frac{1}{2} \sqrt{\eta \zeta},$$
		        \STATE $A_{t+1} := \begin{pmatrix}A_t\\ \widehat{\nabla}_{t}^{\top}\end{pmatrix},\;\;b_{t+1} := \begin{pmatrix}b_t\\\beta_t\end{pmatrix},\;\;k_{t+1} = k_t + 1$.
		    \ENDIF
		\ENDFOR
		\STATE $\lambda_T = \argmin\limits_{\lambda \in \{\lambda_0, ..., \lambda_{T-1}\}} d_{\tau}(\lambda)$
		\ENSURE $\lambda_T$.
	\end{algorithmic}
\end{algorithm}

\begin{theorem}\label{thm:vaidya}
Let $\mathcal{B}_{\mathcal{R}_{in}}$ and $\mathcal{B}_{\mathcal{R}}$ be some Euclidean balls of radii $\mathcal{R}_{in}$ and $\mathcal{R}$, respectively, such that $\mathcal{B}_{\mathcal{R}_{in}} \subseteq \Lambda \subseteq \mathcal{B}_{\mathcal{R}}$, and let a number $B>0$ be such that $|d(\lambda) - d(\lambda')| \leq B\ \forall \lambda, \lambda' \in \Lambda$. After $T \geq \frac{2m}{\zeta} \log \left( \frac{m^{1.5} \mathcal{R}}{\gamma \mathcal{R}_{in}} \right) + \frac{1}{\gamma} \log \pmb{\pi}$ iterations Vaidya's method with $\delta$-subgradient for the problem \eqref{problem_vaidya} returns a point $\lambda_T$ such that
\begin{equation}
    d(\lambda_T) - \min_{\lambda \in \Lambda} d(\lambda) \leq \frac{B m^{1.5} \mathcal{R}}{\zeta \mathcal{R}_{in}} \exp \left( \frac{\log \pmb{\pi} - \zeta T}{2m} \right) + \delta,
\end{equation}
where $\zeta>0$ is the parameter of the algorithm.
\end{theorem}
\begin{corollary}
Vaidya's cutting-plane method with $\delta$-subgradient achieves accuracy $\epsilon$ after
\begin{equation}
    T = \left\lceil \frac{2m}{\zeta}
    \log\left((\epsilon - \delta)^{-1}
     \frac{Bm^{1.5}\mathcal{R}}{\zeta \mathcal{R}_{in}}\right) + \frac{\log \pmb{\pi}}{\zeta} \right\rceil,
\end{equation}
provided that $\epsilon > \delta$ and $\epsilon - \delta \leq B$.
\end{corollary}

\section{Supporting lemmas}
In this section we prove several lemmas and propositions used in the proof of Theorem \ref{convergence_theorem}. From now on, we use notation introduced in Sections \ref{sec_prelim}, \ref{sec_converg}. In particular, we are considering the dual problem
\begin{equation}
\label{dual}
    \min_{\lambda \in \R_{+}^{m}} \bigl\{ d_{\tau}(\lambda):= \max_{\pi \in \Pi} \L_{\tau}(\pi, \lambda) \bigr\}.
\end{equation}

The first lemma establishes upper bound on the norm of a minimizer of the dual function.

\begin{lemma}[see also \cite{lanarcpo}]\label{lem_bounded_dual}
Suppose Assumption \ref{assumption_slater} holds. Let $\lambda_{\tau}^*$ be a solution of the dual problem \eqref{dual}. Then
$$
\left\| \lambda_{\tau}^* \right\|_{1} \leq B_{\lambda}:=\frac{r_{0, \max } + \log |\Ac|}{(1-\gamma) \xi}.
$$
\end{lemma}
\begin{proof}
Note that $\mathcal{H}(\pi) \leq \frac{\log |\Ac|}{1-\gamma}$,
\begin{align*}
    &d_{\tau}(\lambda_{\tau}^*) \geq V_0^{\pi_{\xi}}(\rho)
    + \langle \lambda_{\tau}^*, V^{\pi_{\xi}}(\rho) - c \rangle
    + \tau \mathcal{H}(\pi) \geq
    \xi \|\lambda_{\tau}^*\|_1,\\
    &d_{\tau}(\lambda_{\tau}^*) \leq d_{\tau}(\lambda^*) \leq
    d(\lambda^*) + \frac{\tau \log |\Ac|}{1-\gamma} =
    V^{\pi^*}_0(\rho) + \frac{\tau \log |\Ac|}{1-\gamma} \leq \frac{r_{0,max} + \tau \log |\Ac|}{1-\gamma}, \\
    &\|\lambda_{\tau}^*\|_1 \leq 
    \frac{r_{0, \max } + \tau \log |\Ac|}{(1-\gamma) \xi}.
\end{align*}
\end{proof}

Recall that $\Lambda$ is defined as the set
\begin{equation}
    \Lambda := \{ \lambda \in \R^m_+\; |\; \| \lambda \|_1 \leq B_{\lambda} \}.
\end{equation}
The second lemma gives an example of two Euclidean ball, one of which is contained in $\Lambda$ and the other one contains $\Lambda$.

\begin{lemma}\label{ballsy_lemma}
Let $\mathcal{R} := B_{\lambda},\, \mathcal{R}_{in} := \frac{B_{\lambda}}{m + \sqrt{m}}$, then
$\mathcal{B}_{\mathcal{R}_{in}} \subseteq \Lambda \subseteq \mathcal{B}_{\mathcal{R}}$,
with $\mathcal{B}_{\mathcal{R}_{in}}$ being the Euclidean ball of radius $\mathcal{R}_{in}$ centered at the point $\lambda_{in}:= \mathcal{R}_{in} \cdot \mathbf{1}_m$, $\mathcal{B}_{\mathcal{R}}$ being the Euclidean ball of radius $\mathcal{R}$ centered at the origin.
\end{lemma}
\begin{proof}
To prove the first inclusion, observe that for any $\lambda \in \mathcal{B}_{\mathcal{R}_{in}}$ it holds $\lambda \in \R^m_+$, which implies $\| \lambda\|_1 = \sum_{i=1}^m \lambda_i$. Maximization of this sum subject to constraint $\| \lambda - \lambda_{in} \|_2^2 \leq \mathcal{R}_{in}^2$ yields optimal point $\lambda^{(1)}:=\lambda_{in} + \frac{\mathcal{R}_{in}}{\sqrt{m}}\mathbf{1}_m$. Thus, for any $\lambda \in \mathcal{B}_{\mathcal{R}_{in}}$ we have $\lambda \in \R^m_+,\, \| \lambda\|_1 \leq \| \lambda^{(1)} \|_1 = B_{\lambda}\Rightarrow \lambda \in \Lambda$. The second inclusion follows from the inequality $\| \cdot \|_1 \leq \| \cdot \|_2$.
\end{proof}

The following lemma bounds the range of $d_{\tau}(\lambda)$
on $\Lambda$.

\begin{lemma}\label{lem_dual_func_bound}
The dual function $d_{\tau}(\lambda)$ on the set $\Lambda$ satisfies
\begin{equation}
    0 \leq d_{\tau}(\lambda) \leq B_d :=
    \frac{r_{0,max} + \sqrt{m} B_{\lambda} R_{max} + \tau \log |\Ac|}{1-\gamma}.
\end{equation}
\end{lemma}
\begin{proof}
\begin{align*}
    0 \leq d_{\tau}(\lambda) &= 
    \max_{\pi \in \Pi} V_0^{\pi}(\rho) + 
    \langle \lambda, V^{\pi}(\rho) - c \rangle+ \tau \H(\pi) \\
    &\leq \frac{r_{0,max}}{1-\gamma} + \frac{1}{1-\gamma}\|\lambda\|_2 \cdot R_{max}
    + \frac{\tau \log{|\mathcal{A}|}}{1-\gamma}.
\end{align*}
\end{proof}

Now we establish the fact that the dual function 
$d_{\tau}(\lambda)$ is differentiable on $\Lambda$,
and state what its gradient is.

\begin{proposition}
\label{dtau_gradient}
Suppose assumption \ref{assumption_tau_optimal_unique} holds.
Then $d_{\tau}(\lambda)$ is differentiable
for all $\lambda \in \Lambda$, and
\begin{equation}
    \nabla d_{\tau}(\lambda) = V^{\pi_{\tau,\lambda}^*}(\rho) - c,
\end{equation}
where $\pi^*_{\tau, \lambda} := \argmax_{\pi \in \Pi} \L_{\tau}(\pi, \lambda)$.
\end{proposition}
\begin{proof}
We will apply Danskin's theorem
from \cite{DanskinTheorem} (Proposition B.25, (a)) to
$\L_{\tau}(\pi, \lambda)$, which is 
defined on $\Pi \times \mathbb{R}^m$.
Note that $\Pi$ is compact, 
$\L_{\tau}(\cdot, \cdot)$ is continuous, and $\L_{\tau}(\pi, \cdot)$
is linear (and hence convex and differentiable)
for all $\pi \in \Pi$. Then, according
to Assumption \ref{assumption_tau_optimal_unique},
for $\lambda \in \Lambda$, we also have that the maximizer
for
\begin{equation}
    d_{\tau}(\lambda) = \max_{\pi} \L_{\tau}(\pi, \lambda)
\end{equation}
is unique and equal to $\pi_{\tau, \lambda}^*$.
From Danskin's theorem it then follows,
that $d_{\tau}(\lambda)$ is differentiable for all
$\lambda \in \Lambda$, and
\begin{equation}
    \nabla d_\tau(\lambda) = 
    \nabla_{\lambda} \L_{\tau}(\pi_{\tau, \lambda}^*, \lambda)
    =  V^{\pi_{\tau,\lambda}^*}(\rho) - c.
\end{equation}
\end{proof}

The following two lemmas are required to prove that $d_{\tau}(\lambda)$ is smooth, that is, its gradient is Lipschitz continuous.

\begin{lemma}
\label{smoothness_technical}
Set any $\tau > 0$.
Define the following regularized softmax function
for $x \in \mathbb{R}^n$:
\begin{equation}
    S_{\tau}(x)_i = \frac{\exp(x_i/\tau)}{\sum_{j=1}^n\exp(x_j/\tau)}.
\end{equation}
Then for any $x, x' \in \mathbb{R}^n$ it holds that:
\begin{equation}
    \norm{S_{\tau}(x)-S_{\tau}(x')}_1 \leq \frac{1}{\tau}\norm{x-x'}_{\infty}.
\end{equation}
\end{lemma}
\begin{proof}
First notice that $S_{\tau} = \nabla H_{\tau}$, where:
\begin{equation}
    H_{\tau}(x) = \tau\log\left(\sum_{i=1}^n e^{x_i/\tau}\right).
\end{equation}
So we can write:
\begin{align}
\label{stau_estimate}
    \norm{S_{\tau}(x)-S_{\tau}(x')}_1 &=
    \norm{\nabla H_{\tau}(x)- \nabla H_{\tau}(x')}_1 \stackrel{(i)}{\leq}
    \sup_{z \in \mathbb{R}^n} \norm{\nabla^2H_{\tau}(z)(x'-x)}_1 =\\
    &= \sup_{\substack{z \in \R^n, \\
    u \in \R^n, \norm{u}_{\infty} = 1}}
    u^\top\nabla^2H_{\tau}(z)(x'-x)
    \leq\\
    &\leq 
    \norm{x'-x}_{\infty} \cdot
    \sup_{\substack{z \in \R^n, \\
    u \in \R^n, \norm{u}_{\infty} = 1\\
    v \in \R^n, \norm{v}_{\infty} = 1}}
    u^\top\nabla^2H_{\tau}(z)v,
\end{align}
where $(i)$ can be obtained by considering a function 
$W(t) = \nabla H_\tau(x(1-t) + yt)$ and applying Lagrange
mean inequality for it with $l_1$-norm.

Calculate $\nabla^2H_{\tau}(z)$:
\begin{align}
    \pdv{H}{z_i}{z_j} = \frac{\frac{1}{\tau}\exp(z_i/\tau)\delta_{ij}}
    {(\sum_{k=1}^n{\exp(z_k)/\tau})} - 
    \frac{\frac{1}{\tau}\exp(x_i/\tau)\exp(z_j/\tau)}
    {(\sum_{k=1}^n{\exp(z_k)/\tau})^2}.
\end{align}

Fix some $z \in \mathbb{R}^n$. 
Let $a_i = \frac{\exp(z_i/\tau)}{\sum_{k=1}^n \exp(z_k/\tau)}$.
Note that $\sum_{k=1}^n a_k = 1$, and:
\begin{align}
    \nabla^2H_{\tau}(z)_{ij} = \frac{1}{\tau}a_i\delta_{ij} - \frac{1}{\tau}a_ia_j.
\end{align}

Under the supremum, knowing $\norm{u}_{\infty} = \norm{v}_{\infty} = 1$, we have:
\begin{align}
    u^\top\nabla^2H_{\tau}(z)v &= \frac{1}{\tau}
    \left(\sum_{i=1}a_i u_i v_i -
    \sum_{i=1}^n \sum_{j=1}^n a_ia_ju_iv_j\right) =\\
    &=\frac{1}{\tau}\left(
    \sum_{i=1}^n a_iu_i\left(v_i - \sum_{j=1}^n a_jv_j\right)
    \right)
    =
    \frac{1}{\tau}\left(
    \sum_{i=1}^n a_iu_i\left(\sum_{\substack{j=1\\j\neq i}}^n a_j(v_i-v_j)\right)
    \right)
    \leq\\
    &\leq
    \frac{1}{\tau}\sum_{i=1}^n\sum_{\substack{j=1\\j\neq i}}^n a_ia_j\abs{v_i-v_j}
    \leq \frac{1}{\tau}.
\end{align}

Using this in \ref{stau_estimate}, we finally get:
\begin{align}
    \norm{S_\tau(x) - S_\tau(x')}_1 \leq \frac{1}{\tau}\norm{x'-x}_{\infty}.
\end{align}
\end{proof}

The next result is a corrected and enhanced version of Lemma 7 from \cite{lanarcpo} with a bound improved in a factor of two.

\begin{lemma}[Lemma 7 from \cite{lanarcpo}, corrected and enhanced]
\label{regoptpolicy_smoothness}
The optimal policy for regularized MDP is smooth with respect to $\lambda$,
i.e., for all $\lambda, \lambda' \in \mathbb{R}^m_{+}$, we have:
\begin{equation}
    \max\limits_{s \in \Sc}
    \norm{\pi^{*}_{\tau,\lambda}(\cdot\vert s)-\pi^{*}_{\tau,\lambda'}(\cdot\vert s)}_1
    \leq \frac{R_{\mathrm{max}}}{(1-\gamma)\tau}\norm{\lambda-\lambda'}_2.
\end{equation}
\end{lemma}
\begin{proof}
Proof goes same as in \cite{lanarcpo}, except we bound the promised
$l_1$-norm on the left hand side, instead of $l_{\infty}$ norm.

As was proved in \cite{UniqueRegularized}, the regularized optimal policy
can be expressed in terms of its soft Q-function:
\begin{equation}
    \pi^{*}_{\tau,\lambda}(a\vert s) = \frac
    {\exp(Q^{*}_{\tau,\lambda}(s,a)/\tau)}
    {\sum_{a' \in \Ac}\exp(Q^{*}_{\tau,\lambda}(s,a')/\tau)}.
\end{equation}

Fix some $s \in \Sc$. Using \ref{smoothness_technical}
for $x = Q^*_{\tau,\lambda}(s\vert \cdot),
x' = Q^*_{\tau,\lambda'}(s\vert \cdot)
$ and $\tau$, we get:
\begin{align}
    \norm{\pi^*_{\tau,\lambda}(\cdot \vert s)
    - \pi^*_{\tau,\lambda'}(\cdot \vert s)}_1 \leq 
    \frac{1}{\tau}\norm{Q^*_{\tau,\lambda}(s\vert \cdot)
    - Q^*_{\tau,\lambda'}(s\vert \cdot)}_{\infty} \leq
    \frac{1}{\tau}\norm{Q^*_{\tau,\lambda}
    - Q^*_{\tau,\lambda'}}_{\infty}.
\label{start_estimate}
\end{align}
Furthermore,
\begin{align}
    \norm{Q^*_{\tau,\lambda}
    - Q^*_{\tau,\lambda'}}_{\infty} &\leq
    \max\limits_{s \in \Sc, a \in \Ac}
    \norm{Q^*_{\tau,\lambda}(s,a)
    - Q^*_{\tau,\lambda'}(s,a)}_{\infty}
    \leq\\
    &\leq \max\limits_{s \in \Sc, a \in \Ac} 
    \max_{\pi \in \Pi} \abs{Q^\pi_{\tau,\lambda}(s,a)
    - Q^\pi_{\tau,\lambda'}(s,a)}
    \stackrel{(i)}{\leq}\\
    &\leq \frac{R_{\max}}{1-\gamma}\norm{\lambda-\lambda'}_2,
\label{last_norm_estimate}
\end{align}
where $(i)$ is due to
\begin{align}
    \abs{Q^{\pi}_{\tau,\lambda}(s,a) - Q^{\pi}_{\tau,\lambda'}(s,a)}
    &\leq
    \abs{r_{\lambda}(s,a) - r_{\lambda'}(s,a)} + 
    \gamma \mathbb{E}_{s' \sim P(\cdot\vert s,a)}
    \left[\abs{V^{\pi}_{\tau,\lambda}(s') - V^{\pi}_{\tau,\lambda'}(s')}\right]
    \leq\\
    &\leq
    R_{\max}\norm{\lambda-\lambda'}_2 + \gamma
    \cdot \frac{1}{1-\gamma}R_{\max}\norm{\lambda-\lambda'}_2 =\\
    &=\frac{R_{\max}\norm{\lambda-\lambda'}_2}{1-\gamma}.
\end{align}

Substituting \ref{last_norm_estimate} into \ref{start_estimate},
we get the desired result.
\end{proof}

The next proposition specifies the smoothness coefficient for $d_\tau$. 

\begin{proposition}\label{prop_smooth}
Suppose assumptions \ref{assumption_tau_optimal_unique},
\ref{assumption_ergodicity}
hold, then $d_{\tau}(\lambda)$ is $L_{d}$-smooth:
$$
\left\|\nabla d_{\tau}(\lambda)-\nabla d_{\tau}\left(\lambda^{\prime}\right)\right\|_{2} \leq L_{d}\left\|\lambda-\lambda^{\prime}\right\|_{2}, \quad \forall \lambda, \lambda^{\prime} \in \Lambda,
$$
where $L_{d}=\frac{R_{\max }^{2} L_{\beta}}{(1-\gamma)^{2} \tau}$, $L_{\beta}:=\left\lceil\log _{\beta}\left(C_{M}^{-1}\right)\right\rceil+(1-\beta)^{-1}+1$.
\end{proposition}

\begin{proof}
See proof in \cite{lanarcpo} (Proposition 1),
with $\mu = 0$. Instead of Lemma 7
in it, our version 
\ref{regoptpolicy_smoothness} can be used, making sure the needed 
inequality is correct and improving $L_d$ by a factor of 2.
\end{proof}


The following two lemmas provide a bound on optimality gap and constraint violation in terms of the dual function.
\begin{lemma}
\label{Lemma1}
Suppose Assumption \ref{assumption_tau_optimal_unique} holds and let $\lambda \in \Lambda$, then
\begin{gather}
    V_{0}^{*}(\rho) - V_{0}^{\pi_{\tau,\lambda}^*}(\rho) \leq \bigl\langle \lambda, \nabla d_{\tau}(\lambda) \bigr\rangle + \tau \mathcal{H}\left(\pi_{\tau,\lambda}^*\right),\\
    \bigl\| [ c - V^{\pi_{\tau,\lambda}^*}(\rho) ]_+ \bigr\|_2 = \bigl\| [ -\nabla d_{\tau}(\lambda) ]_+ \bigr\|_2.
\end{gather}
\end{lemma}
\begin{proof}
\begin{multline*}
    V_{0}^{\pi_{\tau,\lambda}^*}(\rho)+\bigl\langle \lambda, V^{\pi_{\tau,\lambda}^*}(\rho)-c\bigr\rangle + \tau \mathcal{H}(\pi_{\tau,\lambda}^*) = \max_{\pi \in \Pi} \L_{\tau}(\pi, \lambda) \\
    \geq \L_{\tau}(\pi^*, \lambda) = V_{0}^{*}(\rho) + \bigl\langle \lambda, V^{\pi^*}(\rho)-c\bigr\rangle + \tau \mathcal{H}(\pi^*).
\end{multline*}
The inequalities $\mathcal{H}(\pi^*) \geq 0$, $V^{\pi^*}(\rho) \geq c$, $\lambda \in \R_+^m$ imply $\L_{\tau}(\pi^*, \lambda) \geq V_{0}^{*}(\rho)$, hence
$$
    V_{0}^{*}(\rho) - V_{0}^{\pi_{\tau,\lambda}^*}(\rho) \leq \bigl\langle \lambda, V^{\pi_{\tau,\lambda}^*}(\rho)-c\bigr\rangle + \tau \mathcal{H}(\pi_{\tau,\lambda}^*).
$$
Using the result of Proposition \ref{dtau_gradient}, we get
$$
    \nabla d_{\tau}(\lambda) = V^{\pi_{\tau,\lambda}^*}(\rho) - c,
$$
which finishes the proof.
\end{proof}



\begin{lemma}\label{Lemma2}
Suppose Assumption \ref{assumption_tau_optimal_unique} holds, and let $\lambda \in \Lambda$, 
then
\begin{gather}
    \label{grad_bound}
    \bigl\| \negrelu{\nabla d_{\tau}(\lambda)} \bigr\|_2^2 \leq 2 L_d (d_{\tau}(\lambda) - d_{\tau}^*),\\
    \label{dot_prod_bound}
    \bigl\langle \lambda, \nabla d_{\tau}(\lambda) \bigr\rangle \leq B_{\lambda} \sqrt{2m L_d (d_{\tau}(\lambda) - d_{\tau}^*)} + 2 (d_{\tau}(\lambda) - d_{\tau}^*),
\end{gather}
where $d_{\tau}^*$ is the optimal value in the dual problem \eqref{dual}, $L_d$ is the smoothness constant of $d_{\tau}$.
\end{lemma}

\begin{proof}
Denote $a:=\nabla d_{\tau}(\lambda)$ and
\begin{equation*}
    N := \{ 1, \ldots, n \},\; I := \{ i \in N:\, a_i \geq 0 \},\; I' := N \setminus I.
\end{equation*}
Moreover, for any vector $b \in \R^m$, define $b_I$ to be the vector with components
$$
    ( b_I )_i := \left\{\begin{array}{cc}
        b_i & \text{if } i \in I, \\
        0 & \text{otherwise}.
    \end{array}\right.\quad
$$

Smoothness implies for any $\lambda' \in \Lambda$
\begin{equation}
    \label{lip}
    \langle a, \lambda - \lambda' \rangle - \frac{L_d}{2} \| \lambda - \lambda' \|_2^2 \leq d_{\tau}(\lambda) - d_{\tau}(\lambda').
\end{equation}
Pick $\lambda' := \relu{\lambda - \frac{1}{L_d} a}$, then it's sufficient to prove that
\begin{equation}\label{grad_norm_bound}
    \frac{1}{2L_d} \bigl\| \negrelu{a} \bigr\|_2^2 \leq \langle a, \lambda - \lambda' \rangle - \frac{L_d}{2} \| \lambda - \lambda' \|_2^2,
\end{equation}
and the first result of the lemma will follow from \eqref{lip}. Using the notation introduced above, we write
\begin{equation*}
    a = a_I + a_{I'}\, \text{ with }\, a_I, (-a_{I'}) \in \R^m_+,\qquad \lambda = \lambda_I + \lambda_{I'}.
\end{equation*}
The vector $\lambda'$ can now be expressed as
\begin{equation*}
    \lambda' = \relu{ \lambda_I + \lambda_{I'} - \frac{1}{L_d} a_I - \frac{1}{L_d} a_{I'} } = \relu{ \lambda_I - \frac{1}{L_d} a_I } + \lambda_{I'} - \frac{1}{L_d} a_{I'},
\end{equation*}
because $I \cap I' = \emptyset$ and $\lambda_{I'} - \frac{1}{L_d} a_{I'}\, \in\, \R^m_+$. The value $\lambda - \lambda'$ writes as
\begin{equation*}
    \lambda - \lambda' = \lambda_I - \relu{ \lambda_I - \frac{1}{L_d} a_I } + \frac{1}{L_d} a_{I'}.
\end{equation*}
The two terms in the right-hand side of \eqref{grad_norm_bound} are equal to
\begin{equation*}
    \langle a, \lambda - \lambda' \rangle = \left\langle a_I, \lambda_I - \relu{ \lambda_I - \frac{1}{L_d} a_I } \right\rangle + \frac{1}{L_d} \| a_{I'} \|_2^2,
\end{equation*}
and
\begin{equation*}
    \frac{L_d}{2} \| \lambda - \lambda' \|^2_2 = \frac{L_d}{2} \bigl\| \lambda_I - \relu{ \lambda_I - \frac{1}{L_d} a_I } \bigr\|_2^2 + \frac{1}{2L_d} \| a_{I'} \|_2^2.
\end{equation*}
Thus, the right-hand side of \eqref{grad_norm_bound} writes as
\begin{equation*}
    \frac{1}{2L_d} \| a_{I'} \|_2^2 + \underbrace{\left\langle a_I, \lambda_I - \relu{ \lambda_I - \frac{1}{L_d} a_I } \right\rangle - \frac{L_d}{2} \bigl\| \lambda_I - \relu{ \lambda_I - \frac{1}{L_d} a_I } \bigr\|_2^2}_{=: \psi}.
\end{equation*}
To observe that $\psi \geq 0$, consider the sets
\begin{equation*}
    J := \left\{ i \in I:\, \lambda_i \geq \smallfrac{1}{L_d} a_i \right\},\; J' := I \setminus J,
\end{equation*}
then
\begin{equation*}
    \relu{ \lambda_I - \frac{1}{L_d} a_I } = \lambda_J - \frac{1}{L_d} a_J\quad \text{and}\quad \lambda_I - \relu{ \lambda_I - \frac{1}{L_d} a_I } = \lambda_{J'} + \frac{1}{L_d} a_J.
\end{equation*}
Thus,
\begin{equation*}
    \psi = \left\langle a_I - \smallfrac{L_d}{2}\left( \lambda_{J'} + \smallfrac{1}{L_d} a_J \right),\, \lambda_{J'} + \smallfrac{1}{L_d} a_J \right\rangle = \frac{1}{2} \left\langle a_I + L_d \left( \smallfrac{1}{L_d} a_{J'} - \lambda_{J'} \right),\, \lambda_{J'} + \smallfrac{1}{L_d} a_J \right\rangle,
\end{equation*}
which is a nonnegative value as a scalar product of vectors with nonnegative components. To sum up,
\begin{equation*}
    \langle a, \lambda - \lambda' \rangle - \frac{L_d}{2} \| \lambda' - \lambda \|_2^2 = \frac{1}{2L_d} \| a_{I'} \|_2^2 + \psi \geq \frac{1}{2L_d} \| a_{I'} \|_2^2,
\end{equation*}
and the first result of the lemma follows from \eqref{lip} since $d_{\tau}(\lambda') \geq d_{\tau}^*$.

The left-hand side of the inequality \eqref{dot_prod_bound} equals
\begin{equation}\label{scal_prod}
    \langle \lambda, a \rangle = \langle \lambda_J, a_J \rangle + \langle \lambda_{J'}, a_{J'} \rangle + \langle \lambda_{I'}, a_{I'} \rangle.
\end{equation}
The last term is non-positive. Let us bound the first two. Put $\lambda':=\lambda - \frac{1}{L_d} a_J$ and observe that $\lambda' \in \R_+^m$ due to the definition of $J$. Moreover, $\lambda' \in \Lambda$ since $\|\lambda'\|_1 \leq \|\lambda\|_1$. The right-hand side of \eqref{lip} writes as
\begin{equation}\label{norm_J}
    \langle a, \lambda - \lambda' \rangle - \frac{L_d}{2} \| \lambda - \lambda' \|^2_2 = \bigl\langle a, \smallfrac{1}{L_d} a_J \bigr\rangle - \frac{1}{2L_d} \| a_J \|^2_2 = \frac{1}{2L_d} \| a_J \|^2_2.
\end{equation}
Relations \eqref{lip} and \eqref{norm_J} yield
\begin{equation}
    \label{norm_J_bound}
    \| a_J \|_2 \leq \sqrt{2L_d (d_{\tau}(\lambda) - d_{\tau}(\lambda'))} \leq \sqrt{2L_d (d_{\tau}(\lambda) - d_{\tau}^*)}.
\end{equation}
Cauchy–Bunyakovsky–Schwarz inequality, condition $\| \lambda \|_2 \leq \sqrt{m} \| \lambda \|_1 \leq \sqrt{m} B_{\lambda}$ and bound \eqref{norm_J_bound} imply
\begin{equation}\label{first_term_bound}
    \langle \lambda_J, a_J \rangle \leq \| \lambda_J \|_2 \cdot \| a_J \|_2 \leq B_{\lambda} \sqrt{2m L_d (d_{\tau}(\lambda) - d_{\tau}^*)}.
\end{equation}
The bound for the second term in the right-hand side of \eqref{scal_prod} can be obtained by putting $\lambda':=\lambda - \lambda_{J'}$. Indeed,
\begin{multline*}
    \langle a, \lambda - \lambda' \rangle - \frac{L_d}{2} \| \lambda - \lambda' \|^2_2 = \langle a, \lambda_{J'} \rangle - \frac{L_d}{2} \| \lambda_{J'} \|^2_2 = \frac{1}{2} \langle a_{J'}, \lambda_{J'} \rangle + \frac{L_d}{2} \left( \frac{1}{L_d} \langle a_{J'}, \lambda_{J'} \rangle - \| \lambda_{J'} \|^2_2 \right) \\
    = \frac{1}{2} \langle a_{J'}, \lambda_{J'} \rangle + \frac{L_d}{2} \bigl\langle \smallfrac{1}{L_d} a_{J'} - \lambda_{J'}, \lambda_{J'} \bigr\rangle \geq \frac{1}{2} \langle a_{J'}, \lambda_{J'} \rangle,
\end{multline*}
where the last inequality is due to the definition of $J'$. Thus,
\begin{equation}
    \label{second_term_bound}
    \langle a_{J'}, \lambda_{J'} \rangle \leq 2 (d_{\tau}(\lambda) - d_{\tau}^*),
\end{equation}
and the second result of the Lemma follows from \eqref{scal_prod}, \eqref{first_term_bound} and \eqref{second_term_bound}.

\end{proof}


\section{Proof of Theorem \ref{convergence_theorem}}\label{main_proof}
In this section, we will write for the shortness of notation $V_0^\pi := V_0^\pi(\rho),V^\pi := V^\pi(\rho)$, and so on, still taking the dependence on $\rho$ into account. Recall that $\pi_T$ is the output of the Algorithm \ref{alg:vaidya} after $T$ iterations, $V_{0}^{*}$ and $d_{\tau}^*$ are optimal values in the primal \eqref{optprob} and dual \eqref{dual_probl} problems, respectively, and the following notation is used
\begin{align}
    V^{\pi}_{\tau, \lambda}
    &:= V_0^\pi 
    + \sum_{i=1}^m \lambda_i V_i^{\pi}
    + \tau\Hc(\pi), \nonumber \\
    V^{*}_{\tau, \lambda}
    &:= \max_{\pi \in \Pi} V^{\pi}_{\tau, \lambda}, \nonumber \\
    \pi^*_{\tau, \lambda} &:= \argmax_{\pi \in \Pi} \L_{\tau}(\pi, \lambda) \equiv \argmax_{\pi \in \Pi} V^{\pi}_{\tau, \lambda}.\label{NPG_probb}
\end{align}
Plan of the proof is as follows.
\begin{enumerate}
    \item We describe how to apply Theorem \ref{thm:vaidya} about convergence of Vaidya's method with $\delta$-subgradient (Algorithm \ref{alg:og_vaidya}) to our proposed Algorithm \ref{alg:vaidya} for the dual problem.\label{proof_pt1}
    \item We express the values $V_{0}^{*}-V_{0}^{\pi^*_{\tau, \lambda_T}}$ and $\bigl\| [ c - V^{\pi^*_{\tau, \lambda_T}} ]_+ \bigr\|_2$ through the optimality gap of the dual problem $d_{\tau}(\lambda_T) - d_{\tau}^*$. That is, we estimate the optimality gap and constraint violation as if the NPG could solve the problem \eqref{NPG_probb} exactly for the last iterate $\lambda_T$.\label{proof_pt2}
    \item We estimate the values $V_{0}^{*}-V_{0}^{\pi_T}$ and $\bigl\| [ c - V^{\pi_T} ]_+ \bigr\|_2$ using the results from the previous step and the convergence rate of Vaidya's algorithm.\label{proof_pt3}
\end{enumerate}

To begin part \ref{proof_pt1} of the proof, observe that the proposed Algorithm \ref{alg:vaidya} is a special case of Vaidya's method with $\delta$-subgradient (Algorithm \ref{alg:og_vaidya}) if the value $-\widehat{\nabla}_{t} := V^{\pi_t} - c$ from line \ref{line_antigrad} of Algorithm \ref{alg:vaidya} is a $\tilde{\delta}$-subgradient (Definition \ref{def_delta_subgrad}) for some $\tilde{\delta}>0$ which depends on the parameter $\delta$ of NPG. This is the case due to the lemma from page 132 of \cite{polyak1983intro} which we give below keeping notation consistent with the rest of the paper.
\begin{lemma}
    Let
    $$
        d_\tau(\lambda) := \max_{\pi \in \Pi} \L_{\tau}(\pi, \lambda),
    $$
    where $\Pi$ is a compact set, $\L_{\tau}(\pi, \lambda)$ is continuous in $\pi, \lambda$ and convex in $\lambda$. 
    Let $\tilde{\pi}$ satisfy for a fixed $\lambda$ the inequality $d(\lambda) - \L_{\tau}(\pi, \lambda) \leq \tilde{\delta}$, then $\partial_{\lambda} \L_{\tau}(\hat{\pi}, \lambda) \in \partial_{\tilde{\delta}} d_{\tau}(\lambda)$.
\end{lemma}

The given lemma shows that $\tilde{\delta}$-optimal policy (in terms
of optimality gap of regularized Lagrangian $\L_{\tau}$) gives us a $\tilde{\delta}$-subgradient for the dual function $d_{\tau}(\lambda)$. According to Theorem \ref{npg_convergence}, the call $\mathrm{NPG}\left(r_0 + \langle \lambda_t, r \rangle, \tau, \delta \right)$ in line \ref{line_NPG} of Algorithm \ref{alg:vaidya} ensures accuracy $\tilde{\delta} = 6 \tau \gamma \delta$. \eg{Additionally, note that the vector $\widehat{\nabla}_{t}$ from line \ref{separ_oracle_vmdp} of the Algorithm \ref{alg:vaidya} satisfies the inequality from line \ref{line_separ_oracle_vaidya} of the Algorithm \ref{alg:og_vaidya}. Indeed, the value $\widehat{\nabla}_{t}^{\top} \lambda_t$ is negative as the sum of negative components of $\lambda_t$, while $\widehat{\nabla}_{t}^{\top} \lambda$ is nonnegative for all $\lambda \in \Lambda$ as a scalar product of two vectors with nonnegative elements.}

Before we can apply Theorem \ref{thm:vaidya} to the proposed algorithm, we need to replace the dual problem $\min_{\lambda \in \R^m_+} d_{\tau}(\lambda)$ with the equivalent one, but on a compact set. This is possible due to Lemma \ref{lem_bounded_dual} which states that the solution $\lambda_{\tau}^*$ of the dual problem satisfies
\begin{equation}\label{dual_bound2}
    \left\| \lambda_{\tau}^* \right\|_{1} \leq B_{\lambda}:=\frac{r_{0, \max } + \log |\Ac|}{(1-\gamma) \xi}.
\end{equation}
Thus, the equivalent formulation is
\begin{equation}\label{compact_dual_problem}
    \min_{\lambda \in \Lambda} d_{\tau}(\lambda)\quad \text{with } \Lambda := \{ \lambda \in \R^m_+\; |\; \| \lambda \|_1 \leq B_{\lambda} \}.
\end{equation}
The only thing left to do is to find the values $\mathcal{R}, \mathcal{R}_{in}$ and $B_d$ such that $\mathcal{B}_{\mathcal{R}_{in}} \subseteq \Lambda \subseteq \mathcal{B}_{\mathcal{R}}$ and $|d_{\tau}(\lambda) - d_{\tau}(\lambda')| \leq B_d\ \forall \lambda, \lambda' \in \Lambda$. Such values are given by Lemmas \ref{ballsy_lemma} and \ref{lem_dual_func_bound}:
\begin{align}
    \mathcal{R} &:= B_{\lambda}, \\
    \mathcal{R}_{in} &:= \frac{B_{\lambda}}{m + \sqrt{m}}, \\
    B_d &:= \frac{r_{0,max} + \sqrt{m} B_{\lambda} R_{max} + \tau \log |\Ac|}{1-\gamma} \stackrel{\eqref{dual_bound2}}{\leq} \left(\xi + \frac{\sqrt{m} R_{max}}{1-\gamma} \right) B_{\lambda}.
\end{align}
We can put
\begin{equation*}
    A_0 := \left[\begin{array}{c}
        -I_m \\
        1
    \end{array}\right],\quad b_0 := \left[\begin{array}{c}
        B_{\lambda} \mathbf{1}_m \\
        m B_{\lambda}
    \end{array}\right],
\end{equation*}
which will correspond to the initial simplex
\begin{equation*}
    P(A_0, b_0)=\bigl\{\lambda \in \mathbb{R}^m: \lambda_{j} \geqslant-\mathcal{R}, j=1,\ldots, m,\ \sum_{j=1}^{m} \lambda_{j} \leqslant m \mathcal{R}\bigr\} \supseteq \mathcal{B}_{\mathcal{R}}.
\end{equation*}
Thus, Theorem \ref{thm:vaidya} applied to the proposed algorithm yields the following convergence estimate for the dual problem \eqref{compact_dual_problem}:
\begin{align}\label{before_eps}
    d_{\tau}(\lambda_T) - d_{\tau}^* &\leq \frac{m^2 (1 + \sqrt{m}) B_{\lambda}}{\zeta} \left(\xi + \frac{\sqrt{m} R_{max}}{1-\gamma} \right) \exp \left( \frac{\log \pmb{\pi} - \zeta T}{2m} \right) + 6 \tau \gamma \delta\\
    &=:\epsilon + 6 \tau \gamma \delta,
\end{align}
where $\epsilon$ denotes the first term of the estimate \eqref{before_eps}.

Part \ref{proof_pt2} of the proof goes as follows. First, we use Proposition \ref{prop_smooth} to state that $d_{\tau}(\lambda)$ is $L_{d}$-smooth with
\begin{equation}\label{smoothnes_const_44}
    L_{d}=\frac{R_{\max }^{2} L_{\beta}}{(1-\gamma)^{2} \tau},\quad \text{where }\, L_{\beta}:=\left\lceil\log _{\beta}\left(C_{M}^{-1}\right)\right\rceil+(1-\beta)^{-1}+1.
\end{equation}
Second, we refer to Lemmas \ref{Lemma1} and \ref{Lemma2} which provide the following bounds:
\begin{align}
    V_{0}^{*} - V_{0}^{\pi^*_{\tau, \lambda_T}} &\leq \bigl\langle \lambda_T, \nabla d_{\tau}(\lambda_T) \bigr\rangle + \tau \mathcal{H}(\pi^*_{\tau, \lambda_T})\\
    &\leq B_{\lambda} \sqrt{2 m L_d (\epsilon + 6 \tau \gamma \delta)} + 2 (\epsilon + 6 \tau \gamma \delta) + \tau \mathcal{H}(\pi^*_{\tau, \lambda_T}),\label{ideal_NPG} \\
    \bigl\| [ c - V^{\pi^*_{\tau, \lambda_T}} ]_+ \bigr\|_2 &= \bigl\| [ -\nabla d_{\tau}(\lambda_T) ]_+ \bigr\|_2 \leq 2 L_d (\epsilon + 6 \tau \gamma \delta).\label{constr_ideal}
\end{align}

Let us begin part \ref{proof_pt3} of the proof.
First, we bound the value $V_{0}^{\pi^*_{\tau, \lambda_T}} - V_{0}^{\pi_T}$. Recall that $\pi_T := \mathrm{NPG}\left(r_0 + \langle \lambda_T, r \rangle, \tau, \delta \right)$. Therefore, according to Theorem \ref{npg_convergence}, $\pi_T$ satisfies
\begin{equation}\label{NPG_44}
    V_{\tau, \lambda_T}^{*} - V_{\tau, \lambda_T}^{\pi_T} \leq 6 \tau \gamma \delta.
\end{equation}
Furthermore,
\begin{align}
    V_{0}^{\pi^*_{\tau, \lambda_T}} - V_{0}^{\pi_T}
    &= \underbrace{V_{0}^{\pi^*_{\tau, \lambda_T}} + \bigl\langle \lambda_T, V^{\pi^*_{\tau, \lambda_T}} \bigr\rangle + \tau \mathcal{H}(\pi^*_{\tau, \lambda_T})}_{V_{\tau, \lambda_T}^{*}}
    - \underbrace{ \left(V_{0}^{\pi_T} + \bigl\langle \lambda_T, V^{\pi_T} \bigr\rangle + \tau \mathcal{H}(\pi_T) \right)}_{V_{\tau, \lambda_T}^{\pi_T}}  \nonumber \\
    &+ \bigl\langle \lambda_T, V^{\pi_T} - V^{\pi^*_{\tau, \lambda_T}} \bigr\rangle + \tau \left( \mathcal{H}(\pi_T) - \mathcal{H}(\pi^*_{\tau, \lambda_T}) \right)  \nonumber \\
    &\stackrel{\eqref{NPG_44}}{\leq} 6 \tau \gamma \delta + \bigl\langle \lambda_T, V^{\pi_T} - V^{\pi^*_{\tau, \lambda_T}} \bigr\rangle + \tau \left( \mathcal{H}(\pi_T) - \mathcal{H}(\pi^*_{\tau, \lambda_T}) \right). \label{NPG_errr}
\end{align}
The scalar product is bounded by
$$
    \bigl\langle \lambda_T, V^{\pi_T} - V^{\pi^*_{\tau, \lambda_T}} \bigr\rangle \leq \|\lambda_T\|_2 \cdot \left\|V^{\pi_T} - V^{\pi^*_{\tau, \lambda_T}}\right\|_2 \leq \sqrt{m}B_{\lambda} \| \hat{\delta} \|_2.
$$
where $\hat{\delta} := V^{\pi_T}(\rho) - V^{\pi^*_{\tau, \lambda_T}}(\rho)$ is a value controlled by the NPG parameter $\delta$.
Now, the optimality gap can be estimated as follows:
\begin{align*}
    V_{0}^{*} - V_{0}^{\pi_T} =\, & V_{0}^{*} - V_{0}^{\pi^*_{\tau, \lambda_T}} + V_{0}^{\pi^*_{\tau, \lambda_T}} - V_{0}^{\pi_T} \\ \stackrel{\eqref{ideal_NPG}}{\leq}& B_{\lambda} \sqrt{2 m L_d (\epsilon + 6 \tau \gamma \delta)} + 2 (\epsilon + 6 \tau \gamma \delta) + \tau \mathcal{H}(\pi^*_{\tau, \lambda_T}) + V_{0}^{\pi^*_{\tau, \lambda_T}} - V_{0}^{\pi_T} \\
    \stackrel{\eqref{NPG_errr}}{\leq}& B_{\lambda} \sqrt{2 m L_d (\epsilon + 6 \tau \gamma \delta)} + 2 (\epsilon + 9 \tau \gamma \delta) + \tau \mathcal{H}(\pi_T) + \sqrt{m}B_{\lambda} \| \hat{\delta} \|_2 \\
    \stackrel{\eqref{smoothnes_const_44}}{=}& \frac{B_{\lambda} R_{max} \sqrt{2 m L_{\beta}}}{1-\gamma}\sqrt{\epsilon/\tau + 6 \gamma \delta} + 2 (\epsilon + 9 \tau \gamma \delta) + \tau \mathcal{H}(\pi_T) + \sqrt{m}B_{\lambda} \| \hat{\delta} \|_2.
\end{align*}
Note that $\mathcal{H}(\pi_T) \leq \frac{\log |\Ac|}{1-\gamma}$. It is reasonable to balance the terms that are proportional to $\sqrt{\epsilon/\tau}$ and $\tau$ by taking $\tau := \min(1, \sqrt[3]{\epsilon})$.

Also, we can bound $\norm{\hat{\delta}}_2$ via knowing that NPG
approximated $\pi_{\tau, \lambda_T}^*$ by $\pi_T$
with accuracy $\delta$, as stated in \ref{npg_convergence}:
\begin{equation}
    \norm{\pi_T - \pi_{\tau,\lambda_T}^*}_\infty < \delta.
\end{equation}

By Lemma 6 from \cite{lanarcpo}, which can be proved for any
two policies, we have:
\begin{equation}
    \norm{\nu_\rho^{\pi_T} - \nu_\rho^{\pi_{\tau,\lambda_T}^*}}_1
    \leq L_\beta \max\limits_{s \in \Sc}
    \norm{\pi_T(\cdot\vert s) - \pi_{\tau,\lambda_T}^*(\cdot\vert s)}_1
    \leq L_{\beta}\abs{\Ac}\delta.
\end{equation}

Now, we can rewrite $\|\hat{\delta}\|_2$:
\begin{align}
    \|\hat{\delta}\|_2^2 &= 
    \norm{V^{\pi_T}(\rho) - V^{\pi_{\tau,\lambda_T}^*}(\rho)}_2^2 = 
    \sum_{i=1}^m 
    \left(
    \frac{1}{1-\gamma}
    \langle r_i, \nu_\rho^{\pi_T} \rangle
    -
    \frac{1}{1-\gamma}
    \langle r_i, \nu_\rho^{\pi_{\tau,\lambda_T}^*} \rangle
    \right)^2
    \leq\\
    &\leq
    \frac{1}{(1-\gamma)^2}
    \sum_{i=1}^m \langle r_i,  \nu_\rho^{\pi_T} - \nu_\rho^{\pi_{\tau,\lambda_T}^*} \rangle^2
    \leq
    \frac{1}{(1-\gamma)^2}
    \sum_{i=1}^m \left(r_{i,\max} 
    \norm{\nu_\rho^{\pi_T} - \nu_\rho^{\pi_{\tau,\lambda_T}^*}}_1\right)^2
    =\\
    &=
    \frac{\norm{\nu_\rho^{\pi_T} - \nu_\rho^{\pi_{\tau,\lambda_T}^*}}_1^2}
    {(1-\gamma)^2} R_{\max}^2.
\end{align}

So, we get:
\begin{equation}
    \|\hat{\delta}\|_2 \leq 
    \frac{\norm{\nu_\rho^{\pi_T} - \nu_\rho^{\pi_{\tau,\lambda_T}^*}}_1}
    {(1-\gamma)} R_{\max} \leq 
    \frac{L_{\beta}\abs{\Ac}R_{\max}}
    {(1-\gamma)} \delta.
\end{equation}

Plugging this estimate and the choice of $\tau = \min(1, \sqrt[3]{\epsilon})$.
into \eqref{smoothnes_const_44},
we get the desired result for optimality gap \ref{opt_gap}.

The second part of the result, \ref{constr_viol},
we achieve as follows:
\begin{align*}
    c - V^{\pi_T} &= c - V^{\pi^*_{\tau, \lambda_T}} + V^{\pi^*_{\tau, \lambda_T}} - V^{\pi_T}, \\
    &\leq \bigl\| [ c - V^{\pi^*_{\tau, \lambda_T}} ]_+ \bigr\|_2 + \bigl\| V^{\pi_T} - V^{\pi^*_{\tau, \lambda_T}} \bigr\|_2 \\
    &\stackrel{\eqref{constr_ideal}}{\leq} 2 L_d (\epsilon + 6 \tau \gamma \delta) + \| \hat{\delta} \|_2 \\
    &\stackrel{\eqref{smoothnes_const_44}}{=} \frac{2 R_{max} L_{\beta}}{1-\gamma}(\epsilon^{2/3} + 6 \gamma \delta) + \| \hat{\delta} \|_2
    \leq\\
    &\leq \frac{2 R_{max} L_{\beta}}{1-\gamma}(\epsilon^{2/3} + 6 \gamma \delta) + \frac{L_{\beta}\abs{\Ac}R_{\max}}
    {(1-\gamma)} \delta.
\end{align*}
where $\| \hat{\delta} \|_2$ is bounded in the same way as earlier.

Having achieved both bounds \ref{opt_gap}, \ref{constr_viol}, we conclude
the proof.


\subsection{Proof of Corollary \ref{main_coroll}}\label{coroll_proof}

Suppose we need the resulting accuracy to be $\kappa > 0$:
\begin{align}
    V_{0}^{*}(\rho)-V_{0}^{\pi_T}(\rho) &\leq \kappa,\\
    \bigl\| [ c - V^{\pi_T}(\rho) ]_+ \bigr\|_2 &\leq \kappa.
\end{align}

We will find some $T, \delta$ to use for the algorithm, so that 
by setting other parameters as in \ref{convergence_theorem},
we will get $\kappa$-optimal solution by its results.
It is enough to satisfy these inequalities:
\begin{align}
    &\frac{B_{\lambda} R_{max} \sqrt{2 m L_{\beta}}}{1-\gamma} \sqrt{\epsilon^{2/3}+6\gamma\delta} < \kappa / 5,\\
    &2\epsilon < \kappa / 5, \\
    &18\gamma \delta \sqrt[3]{\epsilon} < \kappa / 5,\\
    &\frac{\log |\Ac|}{1-\gamma} \sqrt[3]{\epsilon} < \kappa / 5,\\
    &\sqrt{m} B_{\lambda} \frac{L_{\beta}\abs{\Ac}R_{\max}}{1-\gamma}
    \delta < \kappa / 5,\\
    &\frac{2 R_{max}^2 L_{\beta}}{1-\gamma} (\epsilon^{2/3}+6\gamma\delta)
    < \kappa / 2,\\
    &\frac{L_{\beta}\abs{\Ac}R_{\max}}
    {(1-\gamma)} \delta < \kappa / 2.
\end{align}
To satisfy them, it is enough to satisfy these:
\begin{align}
    &\epsilon^{2/3}+6\gamma\delta < 
    \frac{(1-\gamma)^2\kappa^2}
    {50 B_{\lambda}^2 R_{max}^2 m L_{\beta}},\\
    &\epsilon < \kappa / 10, \\
    &\delta \sqrt[3]{\epsilon} < \kappa / (90\gamma),\\
    &\epsilon < \frac{\kappa^3(1-\gamma)^3}{125 \log^3 |\Ac|}, \\
    &\delta < 
    \frac{\kappa(1-\gamma)}
    {5\sqrt{m} B_{\lambda} L_{\beta}\abs{\Ac}R_{\max}},\\
    &\epsilon^{2/3}+6\gamma\delta
    < \frac{\kappa(1-\gamma)}{4 R_{max}^2 L_{\beta}},\\
    &\delta < \frac{(1-\gamma)\kappa}{2L_{\beta}\abs{\Ac}R_{\max}}.
\end{align}
To satisfy them, it is enough to satisfy these:
\begin{align}
    &\epsilon < \frac{(1-\gamma)^3\kappa^3}
    {1000 B_{\lambda}^3 R_{max}^3 m^{3/2} L_{\beta}^{3/2}},\\
    &\delta < 
    \frac{(1-\gamma)^2\kappa^2}
    {600 \gamma B_{\lambda}^2 R_{max}^2 m L_{\beta}},\\
    &\epsilon < \kappa / 10,\\
    &\delta < \kappa^{2/3} / (90\gamma)^{2/3},\\
    &\epsilon < \kappa / (90\gamma),\\
    &\epsilon < \frac{\kappa^3(1-\gamma)^3}{125 \log^3 |\Ac|}, \\
    &\delta < 
    \frac{\kappa(1-\gamma)}
    {5\sqrt{m} B_{\lambda} L_{\beta}\abs{\Ac}R_{\max}},\\
    &\epsilon
    < \frac{\kappa^{3/2}(1-\gamma)^{3/2}}
    {16\sqrt{2} R_{max}^3 L_{\beta}^{3/2}},\\
    &\delta
    < \frac{\kappa(1-\gamma)}{48 \gamma R_{max}^2 L_{\beta}},\\
    &\delta < \frac{(1-\gamma)\kappa}{2L_{\beta}\abs{\Ac}R_{\max}}.
\end{align}
Or, rewritten shorter,
\begin{align}
    &\epsilon < 
    \min\left(
     \frac{\kappa^3(1-\gamma)^3}
    {1000 B_{\lambda}^3 R_{max}^3 m^{3/2} L_{\beta}^{3/2}},
     \frac{\kappa}{10},
     \frac{\kappa}{90\gamma},
     \frac{\kappa^3(1-\gamma)^3}{125 \log^3 |\Ac|},
     \frac{\kappa^{3/2}(1-\gamma)^{3/2}}
    {16\sqrt{2} R_{max}^3 L_{\beta}^{3/2}}
    \right) =: C_\epsilon, \\
    &\delta < \min\left(
    \frac{\kappa^2(1-\gamma)^2}
    {600 \gamma B_{\lambda}^2 R_{max}^2 m L_{\beta}},
    \frac{\kappa^{2/3}}{(90\gamma)^{2/3}},
     \frac{\kappa(1-\gamma)}
    {5\sqrt{m} B_{\lambda} L_{\beta}\abs{\Ac}R_{\max}},
    \frac{\kappa(1-\gamma)}{48 \gamma R_{max}^2 L_{\beta}},
    \frac{\kappa(1-\gamma)}{2L_{\beta}\abs{\Ac}R_{\max}}
    \right) =: C_\delta.
\end{align}
Now, knowing that $\epsilon$ depends on $T$ as in \ref{dual_unopt},
we need to choose $T$, such that $\epsilon < C_{\epsilon}$:
\begin{align}
    &\epsilon = 
    \frac{2m^2 B_{\lambda}}{\zeta} \left(\xi + \frac{\sqrt{m} R_{max}}{1-\gamma} \right) \exp \left( \frac{\log \pmb{\pi} - \zeta T}{2m} \right)
    < C_{\epsilon},\\
    &\exp \left( \frac{\log \pmb{\pi} - \zeta T}{2m} \right) <
    \frac{\zeta C_\epsilon}
    {2m^2 B_{\lambda} \left(\xi + \frac{\sqrt{m} R_{max}}{1-\gamma} \right)},\\
    &\frac{\log \pmb{\pi} - \zeta T}{2m}  < 
    \log( \frac{\zeta C_\epsilon}
    {2m^2 B_{\lambda} \left(\xi + \frac{\sqrt{m} R_{max}}{1-\gamma} \right)}),\\
    &T > \frac{\log \pmb{\pi}}{\zeta} + \frac{2m}{\zeta}
    \log(\frac
    {2m^2 B_{\lambda} \left(\xi + \frac{\sqrt{m} R_{max}}{1-\gamma} \right)}
    {\zeta C_\epsilon}).
\end{align}
And a sufficient number of NPG iterations needed to achieve $C_\delta$
accuracy of each NPG call can be determined using \ref{npg_convergence}:
\begin{align}
    N_{NPG} \approx \frac{
    \log{2C_1R} + \log{C_\delta^{-1}} + \max(\frac{1}{3}
    \log{C_\epsilon^{-1}}, 0)
    }
    {\log \gamma^{-1}},
\end{align}
where $C_1 \geq \normbig{Q_{\tau}^*(\rho) - Q_{\tau}^{(0)}(\rho)}_\infty,
R \geq \max_{s, a} r(s, a)$ for any MDP on which NPG might be called throughout
execution of the algorithm.

It can be seen that asymptotic is
\begin{align}
    &T = \mathcal{O}\left(\frac{m}{\zeta}
    \log(\frac{m \log \abs{\Ac}}{\zeta\xi(1-\gamma)
    (1-\beta)\kappa})
    \right),\\
     &N_{NPG} = \mathcal{O}\left(\frac{1}{1 - \gamma}
    \log(\frac{m \log \abs{\Ac}}{(1-\gamma)\xi(1-\beta)\kappa})
    \right),
\end{align}
which gives us the result (accuracy $\kappa$ is $\epsilon$ in the statement).

\section{Lemmas for the case of regularized dual variables }\label{dual_reg_append}



Consider the regularized dual problem:
\begin{equation}
\label{regularised}
    \max_{\lambda \in \R^m_+} d_{\tau,\mu}(\lambda):=d_{\tau}(\lambda) + \frac{\mu}{2} \| \lambda \|_2^2.
\end{equation}
The objective is $L_{d,\mu}$-smooth with $L_{d,\mu}:=L_d+\mu$. The proof of the convergence mimics the proof from Appendix \ref{main_proof} but with a better bounds on the value $\left \langle  \nabla d_{\tau}(\lambda),\lambda \right \rangle$ and on the norm of dual variable derived below.
\begin{lemma}
It holds
$$
    \frac{1}{2L_{d,\mu}}\sum_{i:\lambda_i>\frac{1}{L_{d,\mu}}\frac{\partial d_{\tau,\mu}}{\partial \lambda_i}}\left (\frac{\partial d_{\tau,\mu}}{\partial \lambda_i}  \right )^2+\frac{1}{2} \sum_{i:\lambda_i\leq\frac{1}{L_{d,\mu}}\frac{\partial d_{\tau,\mu}}{\partial \lambda_i}}\frac{\partial d_{\tau,\mu}}{\partial \lambda_i}\lambda_i \leq d_{\tau,\mu}(\lambda)-d_{\tau,\mu}(\lambda_*).
$$
\end{lemma}
\begin{proof} 

Define $\widetilde{\lambda}:=\left [ \lambda- \frac{1}{L_{d,\mu}}\nabla d_{\tau,\mu}(\lambda) \right ]_+$. Since $d_{\tau,\mu}$ has a Lipschitz continuous gradient on $\Lambda$, the following implications hold:
\begin{equation}
\begin{aligned}
    d_{\tau,\mu}(\widetilde{\lambda})\leq d_{\tau,\mu}(\lambda)&+\bigl\langle \nabla d_{\tau,\mu}, \widetilde{\lambda}-\lambda \bigr\rangle+\frac{L_{d,\mu}}{2}\left \|\widetilde{\lambda}-\lambda \right \|_2^2= \\
    =d_{\tau,\mu}(\lambda)&- \sum_{i:\lambda_i>\frac{1}{L_{d,\mu}}\frac{\partial d_{\tau,\mu}}{\partial \lambda_i}} \frac{\partial d_{\tau,\mu}}{\partial \lambda_i}\frac{1}{L_{d,\mu}}\frac{\partial d_{\tau,\mu}}{\partial \lambda_i}-\sum_{i:\lambda_i\leq\frac{1}{L_{d,\mu}}\frac{\partial d_{\tau,\mu}}{\partial \lambda_i}}\frac{\partial d_{\tau,\mu}}{\partial \lambda_i}\lambda_i +\\
    &+\sum_{i:\lambda_i>\frac{1}{L_{d,\mu}}\frac{\partial d_{\tau,\mu}}{\partial \lambda_i}}\frac{L_{d,\mu}}{2}\left(\frac{1}{L_{d,\mu}}\frac{\partial d_{\tau,\mu}}{\partial \lambda_i}\right)^2+\sum_{i:\lambda_i\leq\frac{1}{L_{d,\mu}}\frac{\partial d_{\tau,\mu}}{\partial \lambda_i}}\frac{L_{d,\mu}}{2}\lambda^2_i\\
    \leq d_{\tau,\mu}(\lambda)&-\frac{1}{2L_{d,\mu}}\sum_{i:\lambda_i>\frac{1}{L_{d,\mu}}\frac{\partial d_{\tau,\mu}}{\partial \lambda_i}}\left(\frac{\partial d_{\tau,\mu}}{\partial \lambda_i}\right)^2-\frac{1}{2}\sum_{i:\lambda_i\leq\frac{1}{L_{d,\mu}}\frac{\partial d_{\tau,\mu}}{\partial \lambda_i}}\frac{\partial d_{\tau,\mu}}{\partial \lambda_i}\lambda_i.
\label{lemma}
\end{aligned}
\end{equation}
The statement now follows from $d_{\tau,\mu}(\widetilde{\lambda})\geq d_{\tau,\mu}(\lambda_*)$.
\end{proof}

\begin{lemma}
\begin{equation}
    \left \langle  \lambda, \nabla d_{\tau}(\lambda) \right \rangle \leq \frac{L_{d,\mu}}{\mu}(d_{\tau,\mu}(\lambda) -d_{\tau,\mu}(\lambda_*)).
\end{equation}
\end{lemma}
\begin{proof}
From \eqref{lemma} we have:
\begin{equation}
\begin{aligned}
    d_{\tau,\mu}(\lambda)-d_{\tau,\mu}(\lambda_*)\geq\frac{1}{2L_{d,\mu}}\sum_{i:\lambda_i>\frac{1}{L_{d,\mu}}\frac{\partial d_{\tau,\mu}}{\partial \lambda_i}}\left (\frac{\partial d_{\tau,\mu}}{\partial \lambda_i}  \right )^2+\frac{1}{2}&\sum_{i:\lambda_i\leq\frac{1}{L_{d,\mu}}\frac{\partial d_{\tau,\mu}}{\partial \lambda_i}}\frac{\partial d_{\tau,\mu}}{\partial \lambda_i}\lambda_i.
\end{aligned}
\end{equation}
Note that
\begin{equation}
\label{first_sum}
\left (\frac{\partial d_{\tau,\mu}}{\partial \lambda_i}  \right )^2= \left (\frac{\partial d_{\tau}}{\partial \lambda_i} +\mu\lambda_i \right )^2\geq 2\mu\frac{\partial d_{\tau,\mu}}{\partial \lambda_i}\lambda_i,
\end{equation}
\begin{equation}
\label{second_sum}
\frac{\partial d_{\tau,\mu}}{\partial \lambda_i}\lambda_i\geq\frac{\partial d}{\partial \lambda_i}\lambda_i, \ \lambda_i\geq0.
\end{equation}
Then, from \eqref{first_sum},\eqref{second_sum}
\begin{equation}
    d_{\tau,\mu}(\lambda)-d_{\tau,\mu}(\lambda_*)\geq\frac{2\mu}{2L_{d,\mu}}\sum_{i:\lambda_i>\frac{1}{L_{d,\mu}}\frac{\partial d_{\tau,\mu}}{\partial \lambda_i}}\frac{\partial d}{\partial \lambda_i}\lambda_i+\frac{1}{2}\sum_{i:\lambda_i\leq\frac{1}{L_{d,\mu}}\frac{\partial d_{\mu,\tau}}{\partial \lambda_i}}\frac{\partial d}{\partial\lambda_i}\lambda_i\geq\\
    \frac{\mu}{L_{d,\mu}} \left \langle  \nabla d_{\tau}(\lambda),\lambda \right \rangle,
\end{equation}
and we get:
\begin{equation}
    \left \langle  \nabla d_{\tau}(\lambda),\lambda \right \rangle \leq \frac{L_{d,\mu}}{\mu}(d_{\tau,\mu}(\lambda)-d_{\tau,\mu}(\lambda_*)).
\end{equation}
\end{proof}

Regularization by $\mu$:
\begin{lemma}
\begin{equation}
\begin{aligned}
    \left\|\lambda_{*}^{\mu}\right\|_{2}^{2} \leqslant \frac{2}{\mu}\left (d_{\tau}(0)-d_{\tau}\left(\lambda_{*}\right)\right),
\end{aligned}
\end{equation}
where $\lambda^{\mu}_*$ is the solution of \eqref{regularised}.
\end{lemma}
\begin{proof}
From $\frac{\mu}{2}\left\|\lambda-\lambda_*\right\|_{2}^{2}\leq d_{\tau}(\lambda)-d_{\tau}(\lambda_*)$, we get
\begin{equation}
    d_{\tau,\mu}\left(\lambda_{*}^{\mu}\right)=d_{\tau}\left(\lambda_{*}^{\mu}\right)+\frac{\mu}{2}\left\|\lambda_{*}^{\mu}\right\|_{2}^{2} \geqslant d_{\tau}\left(\lambda_{*}\right)+\frac{\mu}{2}\left\|\lambda_{*}^{\mu}\right\|_{2}^{2} \geqslant d_{\tau}\left(\lambda_{*}\right).
\end{equation}
Then 
\begin{equation}
    \frac{\mu}{2}\left\|\lambda_{\mu}^{*}\right\|_{2}^{2}=\frac{\mu}{2}\left\|0-\lambda_{\mu}^{*}\right\|_{2}^{2} \leqslant d_{\tau,\mu}(0)-d_{\tau,\mu}\left(\lambda_{*}^{\mu}\right) \leqslant d_{\tau}(0)-d_{\tau}\left(\lambda_{*}\right),
\end{equation}
where we used that $d_{\tau,\mu}(0)=d_{\tau}(0)$.
\end{proof}

\section{Experimental Parameters}\label{sec:expparam}

\subsection{Environment}
\label{sec:paramA}

The reward and cost functions in our environment are the same as in \cite{lanarcpo}. The agent receives a reward +1 for the end of the lower link being at a height of 0.5 and cost one of 1 when the first link swings at a anticlockwise direction and the agent applies a +1 torque to the actuating joint; it also receives a cost two of 1 when the second link swings at a anticlockwise direction with respect to the first link and the agent applies a +1 torque to the actuating joint. The cost thresholds are 50.

\subsection{Algorithm parameters}
\label{sec:paramB}
The policy networks for all experiments have two hidden layers of sizes 128 with ReLu activation function. We also have value networks with the same architecture and activation functions as the policy networks. Table \ref{tab:alg_params} summarizes the hyperparameters used in our experiments. 

\begin{table}[ht]
\begin{minipage}{\textwidth}
\caption{Hyperparameters in total reward case.}
\label{tab:alg_params}
\begin{tabular*}{\textwidth}{l @{\extracolsep{\fill}} lll}
\hline
Hyperparameter         & VMDP & AR-CPO \\ \hline
Batch size             &  $5100$    &  $5100$      \\
Discount $\gamma$      &  $0.98$   &  $0.98$      \\
Maximum episode length &  $500$   &   $500$     \\
Learning rate          &   $1$   &    $1$    \\
The number of policy optimization steps in NPG subroutine  &  $4$    &    $1$    \\
max\_KL: the parameter that controls the NPG updates&  $0.01$    &    $0.01$     \\
$\eta$: the parameter in the update of $\lambda$ from \cite{lanarcpo}&  N/A    &    $ 0.0003$    \\
$s$ from \cite{lanarcpo} &   N/A   &   $1$     \\
$H$ from \cite{lanarcpo} &   N/A   &   $45$    \\
Entropy regularisation constant $\tau$ &  $0.01$     &  $0$      \\
Regularization coefficient $\mu$  &  $0$     &  $0$      \\
Optimization set radius for $\lambda$  &  $1$     &  N/A    \\
$\eta$ from Algorithm \ref{alg:vaidya} &  $1000$     &  N/A    \\
$\zeta$ from Algorithm \ref{alg:vaidya} &  $10^{-1}$     &  N/A    \\
\hline
\end{tabular*}
\end{minipage}
\end{table}


\subsection{Algorithm parameters for discounted case} \label{sec:additional_expparam}

The most of the parameters are similar to according ones in Table \ref{tab:alg_params}. Values, which differ, are represented in Table \ref{tab:alg_params_add}.

\begin{table}[ht]
\begin{minipage}{\textwidth}
\caption{Changed hyperparameters in discounted case.}
\label{tab:alg_params_add}
\begin{tabular*}{\textwidth}{l @{\extracolsep{\fill}} lll}
\hline
Hyperparameter         & VMDP & AR-CPO \\ \hline
The number of policy optimization steps in NPG subroutine  &  $40$    &    $1$    \\
max\_KL: the parameter that controls the NPG updates&  $0.01$    &    $0.001$     \\
Entropy regularisation constant $\tau$ &  $0.001$     &  $0$      \\
Regularization coefficient $\mu$  &  $0.01$     &  $0.001$      \\
\hline
\end{tabular*}
\end{minipage}
\end{table}
\end{document}